\documentclass[a4paper,10pt]{amsart}

\usepackage{amsmath}
\usepackage{amsfonts}
\usepackage{amssymb}
\usepackage{amsthm}
\usepackage{enumerate}
\usepackage[latin1]{inputenc}
\usepackage{color}
\usepackage[all]{xy}

\renewcommand{\AA}{\mathbb{A}}
\newcommand{\NN}{\mathbb{N}}
\newcommand{\ZZ}{\mathbb{Z}}

\newcommand{\RR}{\mathbb{R}}
\newcommand{\CC}{\mathbb{C}}

\newcommand{\PP}{\mathbb{P}}
\newcommand{\TT}{\mathbb{T}}

\newcommand{\LL}{\mathbb{L}}
\newcommand{\GG}{\mathbb{G}}

\newcommand{\C}{\mathcal{C}}
\newcommand{\D}{\mathcal{D}}

\DeclareMathOperator{\Hom}{Hom}

\DeclareMathOperator{\id}{id}

\DeclareMathOperator*{\colim}{colim}
\DeclareMathOperator*{\hocolim}{hocolim}
\DeclareMathOperator*{\holim}{holim}
\DeclareMathOperator{\spec}{Spec}

\DeclareMathOperator{\sPre}{sPre}

\DeclareMathOperator{\sShv}{sShv}

\newcommand{\psPre}{\sPre_{\cdot}}

\newcommand{\GTop}{{\mathrm{G}\mathcal{T}\mathrm{op}}}
\newcommand{\Top}{{\mathcal{T}\mathrm{op}}}

\newcommand{\op}{{\mathrm{op}}}

\newcommand{\sset}{\mathrm{s}\mathcal{S}\mathrm{et}}
\newcommand{\smk}{\mathcal{S}\mathrm{m}/k}
\newcommand{\Gsmk}{G\mathcal{S}\mathrm{m}/k}
\newcommand{\sm}[1]{\mathcal{S}\mathrm{m}/\mathrm{#1}}
\newcommand{\schk}{\mathcal{S}\mathrm{ch}_\mathrm{k}}
\newcommand{\SHKG}{\mathcal{SH}(k,G)}

\newcommand{\set}{\mathcal{S}\mathrm{et}}
\newcommand{\Nspec}{\mathcal{S}\mathrm{p}^\mathbb{N}}

\renewcommand{\smash}{\wedge}

\newcommand{\tr}{_{\mathrm{tr}}}
\newcommand{\fixed}{_{\mathrm{fixed}}}
\newcommand{\res}{\mathrm{res}}
\newcommand{\ind}{\mathrm{ind}}

\theoremstyle{plain} 
\newtheorem{theorem}{Theorem}[section]
\newtheorem*{intro:prop1}{Proposition \ref{prop:charofA1localequivariantweq}}
\newtheorem*{intro:prop2}{Proposition \ref{prop:eqkthydescent}}
\newtheorem{corollary}[theorem]{Corollary}
\newtheorem{proposition}[theorem]{Proposition}
\newtheorem{lemma}[theorem]{Lemma}
\theoremstyle{definition}
\newtheorem{defi}[theorem]{Definition}
\newtheorem{bsp}[theorem]{Example}
\newtheorem{rem}[theorem]{Remark}
\numberwithin{equation}{section}
\title{Equivariant Motivic Homotopy Theory}
\author{Philip Herrmann}
\address{Universität Hamburg, Bundesstraße 55, 20146 Hamburg}
\email{philip.herrmann@math.uni-hamburg.de}

\begin{document}

\begin{abstract}
 In this paper we study a model structure on a category of schemes with a group action and the resulting unstable and stable equivariant motivic homotopy theories. The new model structure introduced here samples a comparison to the one by Voevodsky and Hu-Kriz-Ormsby. We show that it allows to detect equivariant motivic weak equivalences on fixed points and how this property leads to a topologically convenient behavior of stable equivalences. We also prove a negative result concerning descent for equivariant algebraic K-theory.
\end{abstract}

\maketitle

\section{Introduction}
The study of transformation groups has a long history in many abstract and geometric areas of mathematics, including topology and algebraic geometry. However, recently equivariant matters experienced an increased focus in algebraic topology, not only due to their role in the work of Hill-Hopkins-Ravenel. This trend is also observed in motivic homotopy theory, where a foundational setup for equivariant considerations is provided by \cite{voevodsky2001lectures, Hu2011homotopy}. In this work, we present an alternative account to equivariant motivic homotopy theory, based on slight variation of the Nisnevich-style Grothendieck topology on the category of smooth $G$-schemes over a field. This new topology is build in a way that allows to detect equivariant local weak equivalences on fixed points. More precisely, under rather mild restrictions to the transformation group $G$ (cf.~Remark \ref{rem:forarty}), there exist right adjoint fixed point functors $(-)^H:\Gsmk\to\smk$, for all $H\leq G$, whose left Kan extensions give rise to a family of functors $(-)^H:\sPre(\Gsmk)\to\sPre(\smk)$, such that $f:X\to Y$ in $\sPre(\Gsmk)$ is a local weak equivalence if and only if $f^H:X^H\to Y^H$ is an ordinary (Nisnevich-)local weak equivalence for all $H\leq G$ (cf.~Corollary \ref{cor:charofeqlocaleq}). Further, we show that the usual $\AA^1$ contracting Bousfield localization interacts nicely with respect to the local model structures on both sides and the above implies a characterisation of equivariant $\AA^1$-local weak equivalences in the same terms.

\begin{intro:prop1}
A morphism $f\in \sPre(\Gsmk)$ is an $\AA^1$-local weak equivalence if and only if for all subgroups $H\leq G$ the morphism $f^H$ is an $\AA^1$-local weak equivalence in $\sPre(\smk)$.
\end{intro:prop1}

Following the topological work of Mandell \cite{mandell2004equivariant} and its motivic adaption in \cite{Hu2011homotopy}, we develop some stable equivariant motivic homotopy theory. In particular, we show how motivic analogs of (Lewis-May) fixed points $(-)^H$ and geometric fixed points $\Phi^H$, from genuine motivic $G$-spectra to ordinary $\PP^1$-spectra, both detect equivariant stable weak equivalences, cf.~Propositions \ref{prop:charstableweqfixed} \& \ref{prop:charstableweqgeometric}.

Finally, we investigate the descent property for equivariant algebraic K-theory and conclude that the topology under investigation does not allow equivariant algebraic K-theory to satisfy descent. We notice that by the same argument also the isovariant topology of \cite{serpe2010descent} does not allow descent for equivariant algebraic K-theory - in contrast to \cite[Theorem 4.2]{serpe2010descent}. These results have to be seen in correlation to \cite{krishna2012nisnevich} where it is shown that the equivariant Nisnevich topology from \cite{voevodsky2001lectures, Hu2011homotopy} allows K-theory descent.

\subsection*{Acknowledgements}
This paper is essentially the first part of my thesis \cite{herrmann2012stable}. I want to thank my advisor Oliver R\"ondigs for the great support in that time. Also I want to thank Jeremiah Heller, Paul Arne Østvær, Florian Strunk and Ben Williams for the many helpful discussions and suggestions that influenced this paper.

\section{Equivariant Grothendieck Topologies}

In \cite{voevodsky2001lectures} Voevodsky first defined a Nisnevich-style topology on the category of $G$-equivariant quasi-projective schemes. This approach was taken up by others and has since then been developed in several slightly different contexts, e.g.~\cite{Hu2011homotopy, krishna2012nisnevich,HOV}. Originally, Voevodsky defined the \emph{equivariant Nisnevich topology} as generated by étale equivariant maps which have an equivariant splitting sequence. After a discussion of concepts of stabilizers and fixed-points, we will rephrase the following definition in Lemma \ref{lemma:equivalenteqnis}.

\begin{defi}\label{def:equiNistop}
 An étale equivariant morphism $f:X\to Y\in\Gsmk$ is a covering in the equivariant Nisnevich topology if there is a sequence of $G$-invariant closed subvarieties
 $$
 \emptyset=Y_{n+1}\subsetneq Y_n \subsetneq \ldots \subsetneq Y_0=Y,
 $$
 such that $f$ has an equivariant section on $Y_j\setminus Y_{j+1}$.
\end{defi}

The main focus in this work will be on the following alternative topology.

\begin{defi}\label{def:HNistop}
 A morphism $f:X\to Y\in\Gsmk$ is a covering in the fixed-point Nisnevich topology, if $f^H:X^H\to Y^H$ is an ordinary Nisnevich covering in $\smk$, for all $H\leq G$.
\end{defi}

\begin{rem}
 \begin{enumerate}[a)]
  \item Instead of considering all subgroups, we could just insist on Nisnevich covers for a family $\mathcal{F}$ of subgroups. We would call the resulting topology the $\mathcal{F}$-fixed-point Nisnevich topology or just the $\mathcal{F}$-Nisnevich topology. For the family $\mathcal{F}=\mathcal{A}ll$ consisting of all $H\leq G$, we abbreviate the notation speak of the $H$-Nisnevich topology in the following. 
  \item After a recollection of definitions and properties of isotropy and fixed-point functors for schemes in the following subsection, we will have a closer look at these two topologies and have a detailed comparison result in subsection \ref{subsec:comparison}.
 \end{enumerate}

\end{rem}

\subsection{Isotropy and Fixed-Points}\label{subsec:isotropyandfixed}

Now we introduce two concepts of stabilizers. Their difference will be
responsible for many distinctions in the following and will finally explain the difference between the two topoloies introduced above - the equivariant Nisnevich topology and the fiexed point Nisnevich topology.

\begin{defi}
 Let $G$ be a group scheme acting on a scheme $X$ and $x:\kappa(x)\to X$ be a
point of $X$. The scheme theoretic stabilizer $G_x$ is defined by the pullback
diagram
\begin{equation}\label{eq:stabilizer}
\begin{xy}
\xymatrix{
G_x\ar[r]\ar[d]& G\times X\ar[d]^{(\alpha_X,\mathrm{pr}_X)}\\
\kappa(x)\ar[r]^{\Delta\circ x}&X\times X.
}
\end{xy}
\end{equation}
\end{defi}

In general, an action $\alpha_X:G\times X\to X$ of a group object $G$ on some
object $X$ is called free if the morphism
$$
(\alpha_X,\mathrm{pr}_X):G\times X\to X\times X
$$
is a monomorphism. By the characterization of locally finite type monomorphisms
\cite[Proposition 17.2.6]{EGA4-4} this implies that an action in $\smk$ is free
if and only if all isotropy groups are trivial in the sense that
$G_x\xrightarrow{\cong}\spec(\kappa(x))$ is an isomorphism for all $x\in X$.

There is a forgetful functor $U:\Gsmk\to|G|\Top$. As we only consider finite
constant group schemes $G$, we disregard the difference between $G$ and its
underlying space $|G|$ here. Applying $U$ to the diagram \eqref{eq:stabilizer}
we obtain a morphism $i:G_x\to S_x$ into the pullback in $G\Top$:
\begin{equation*}
\begin{xy}
\xymatrix{
G_x\ar@{-->}[dr]&&\\
&S_x\ar[r]\ar[d]& G\times UX\ar[d]\\
&\ast\ar[r]^{Ux}&UX\times UX.
}
\end{xy}
\end{equation*}
where $S_x$ is the set theoretic stabilizer of $x$.

\begin{lemma}
 Let $G$ be a finite constant group acting on a scheme $X$ and let $x\in X$.
Then there is an inclusion of subgroups $ G_x\leq S_x\leq G$.
\end{lemma}
\begin{proof}
We know that for an element $x$ in the underlying set $UX$ of the scheme $X$ we
have
\begin{align*}
S_x&=\{ g\in G \mid \text{the set map } g:UX\to UX \text{ satisfies } gx=x\}\\
\intertext{
and in the same way we can describe (the underlying set of) $G_x$ as}
G_x&=\{g\in S_x\mid \text{the induced morphism } g:\kappa(x)\to\kappa(x) \text{
equals } \id_{\kappa(x)}\}.
\end{align*}
\end{proof}

\begin{bsp}
Let $L:k$ be a Galois extension and consider the Galois action of $G:=\mathrm{Gal}(L:k)$ on $\spec(L)$. Then the scheme theoretic stabilizer $G_\ast$ of the unique point $\ast$ in $\spec(L)$ is trivial while the set theoretic isotropy $S_\ast$ is all of $G$ in this case. For the induced action of $G$ on $\spec(\mathcal{O}_L)$, the scheme theoretic stabilizer $G_\mathfrak{p}$ of a point $\mathfrak{p}\in\spec(\mathcal{O}_L)$ recovers the inertia group of $\mathfrak{p}$ while $S_\mathfrak{p}$ gives the decomposition group of $\mathfrak{p}$. 
\end{bsp}

Let $k$ be a field of characteristic $0$ and let $\Gsmk$ be the category of
$G$-equivariant separated smooth $k$-schemes with $G$-equivariant morphisms.
Much power in classical equivariant topology is obtained from adjunctions
connecting equivariant to non-equivariant questions, e.g.~the two adjunctions
with the functor from spaces to $G$-spaces which adds a trivial $G$-action. Due
to the usual problems with quotients in algebraic geometry it seems to be
difficult to carry both of the mentioned adjunctions to a motivic setup. Therefore, we decide to build up our theory with a focus on an adjunction analogous to the classical adjunction
$$
(-)_{\mathrm{trivial}}:\Top\rightleftarrows\GTop:(-)^G.
$$

For any $k$-scheme $X$ there is the trivial $G$-scheme 
\begin{equation}\label{eq:trivialGaction}
X\tr=(X,G\times X\xrightarrow{\pi_X} X) 
\end{equation}
over $k$. Mapping $X$ to $X\tr$ gives embeddings $\schk \subset G\schk$ and
$\smk \subset\Gsmk$. For $X\in G\schk$, we define the functor
$$
h_{X^G}:\schk^\op\to\set, Y\mapsto \Hom_{G\schk}(Y\tr,X).
$$
It is natural to ask for the representability of $h_{X^G}$ and one is inclined
to denote a representing object by $X^G$. The following theorem answers this
question and supports the notation.
\begin{theorem}\label{thm:repfixedpoints}
Let $G$ be a finite constant group scheme over $k$ and let $X\in G\schk$. Then
there exists a $G$-invariant closed subscheme $X^G$ of $X$ with a trivial
$G$-action, representing $h_{X^G}$.
\end{theorem}
\begin{proof}
 Let $\{U_i\}_{i\in I}$ be the family of all closed $G$-invariant subschemes of
$X$ on which $G$ acts trivially and let $J_i$ be the quasi-coherent ideal of
$\mathcal O_X$ corresponding to $U_i$. Let $J:=\cap_i J_i$ be the intersection
of $\mathcal O_X$ modules and denote by $X^G$ the closed subscheme of $X$
corresponding to the ideal sheaf $J$. Then $X^G$ is $G$-invariant and has a
trivial $G$-action as it is shown in \cite[Theorem 2.3]{fogarty1973fixed}.
\end{proof}

\begin{rem}\label{rem:forarty}
Theorem \ref{thm:repfixedpoints} has a notable history. It is stated in more
general terms as \cite[Exp. VIII, Théorème 6.4]{SGA3}. Fogarty still tried to
loose the assumptions on $G$ in \cite[Theorem 2.3]{fogarty1973fixed}, but his
published proof contains a gap which can not be closed, as shown in
\cite{wright1976flat}. However, in this special case of a finite constant group
scheme Fogarty's proof also holds.
\end{rem}

\begin{lemma}
 Let $G$ be a finite constant group scheme over $k$ and let $X\in \Gsmk$. Then $X^G$ is a smooth $k$-scheme and thus we have an adjunction
 \begin{equation}\label{eq:schemeadj}
  tr:\smk\rightleftarrows\Gsmk:(-)^G.
 \end{equation}
\end{lemma}

\begin{proof}
This follows essentially from Luna's slice theorem \cite[Théorème du slice étale, p.97]{luna1973slices} (cf.~\cite[(1)]{Hu2011homotopy}).
\end{proof}

We obtain similar adjunctions for all subgroups $H\leq G$ as a composition
$$
\smk\rightleftarrows H\smk \rightleftarrows \Gsmk:(-)^H,
$$
where the left adjunction is as in \eqref{eq:schemeadj} and the right adjunction is given by restricting the $G$ action to an $H$ action. This gives the family $\{(-)^H\}_{H\leq G}$ of fixed-point functors we have used in Definition \ref{def:HNistop} to define the $H$-Nisnevich topology.

\subsection{Comparison}\label{subsec:comparison} Now that we have recalled the essential concepts for a
distinction of the equivariant Nisnevich topology and the alternative
$H$-Nisnevich topology, we will rephrase these topologies in terms focusing
on pointwise isotropy groups. This will allow to describe a relation of the two
topologies in Corollary \ref{cor:eqNisisHNis}.

The following lemma is \cite[Proposition 3.5]{HOV} and gives a good collection of the equivalent definitions of the equivariant Nisnevich topology from the literature.
\begin{lemma}\label{lemma:equivalenteqnis}
 Let $f:X\to Y$ be an étale morphism in $\Gsmk$. The following are equivalent:
 \begin{enumerate}
  \item $f$ is a covering in the equivariant Nisnevich topology.
  \item $f$ is a covering in the topology generated by the cd-structure with squares
\begin{equation*}
\begin{xy}
\xymatrix{
A\ar[r]\ar[d]& C\ar[d]^{p}\\
B\ar[r]^{i}&D,
}
\end{xy}
\end{equation*}
  where $i$ is an open inclusion and $p$ is étale and restricts to an isomorphism $p^{-1}(D\setminus B)_{red}\cong (D\setminus B)_{red}$.
  \item For every $y\in Y$ there is an element $x\in X$ such that $f$ induces an isomorphism between the residue class field $\kappa(x)$ and $\kappa(y)$ and between the set-theoretic isotropy groups.
 \end{enumerate}
\end{lemma}

\begin{lemma}\label{lemma:setschemestabilizer} Let $G$ be a finite constant
group and let $f:X\to Y$ be an étale morphism in $\Gsmk$ such that for all $y\in
Y$ there is an element $x\in X$ with
$f^*:\kappa(y)\xrightarrow{\cong}\kappa(x)$. If there is such an $x$ with the
additional property that $S_x=S_{f(x)}$ then $f$ induces an isomorphism of the
respective scheme theoretic stabilizers.
\end{lemma}
\begin{proof}
Since $f$ is equivariant, we have an inclusion of the underlying subgroups
$G_x\leq G_{f(x)}$ for all $x\in X$. Let $y\in Y$ and $x\in X$ be as above and
let $g$ be an element in the underlying set of $G_y$. From the assumptions we
know that $g$ is then also an element in the set theoretic stabilizer $S_x$.
Consider the commutative square
\begin{equation*}
\begin{xy}
\xymatrix{
\mathcal{O}_{Y,y}\ar[r]^{g_y=\id}\ar[d]^{f_x}&\mathcal{O}_{Y,y}\ar[d]^{f_x}\\
\mathcal{O}_{X,x}\ar[r]^{g_x}&\mathcal{O}_{X,x}.
}
\end{xy}
\end{equation*}
We need to see that the action $g_x$ induced by $g$ on the local ring of $X$ at
$x$ is trivial, i.e.~$g_x=\id$, to conclude that the underlying subgroups $G_x$
and $G_y$ coincide. Since $f_x$ induces an isomorphism on residue fields, it
follows from Nakayama's Lemma that $f_x$ is itself surjective. So, $f_x$ is an epimorphism
and we cancel it in $f_x=g_x\circ f_x$ to obtain $g_x=\id$ and hence $G_x=G_y$
for the underlying subgroups of $G$. Finally, we may again apply that $f$
induces an isomorphism between $\kappa(y)$ and $\kappa(x)$ to obtain that $f$
also induces an isomorphism
$$
G_x=|G_x|\times \kappa(x)\xrightarrow{\cong}|G_y|\times \kappa(y)=G_y
$$
of the scheme theoretic isotropy groups.
\end{proof}

\begin{lemma}
 A morphism $f:X\to Y$ in $\Gsmk$ is a cover in the $H$-Nisnevich topology (an $H$-cover) if and only if $f$ is étale (as a morphism of schemes), for every point $y$ in $Y$ there is a point $x$ in $X$, such that $f$ induces an isomorphism of residue fields, and
\begin{equation*}
\text{also induces an isomorphism } G_x\xrightarrow{\cong} G_y \text{ of scheme theoretic stabilizers.}\tag{*}
\end{equation*}
\end{lemma}
\begin{proof}
First, assume that $f:X\to Y\in\Gsmk$ is a morphism such that $f^e$ is Nisnevich in $\smk$ and $f$ induces an isomorphism on scheme theoretic isotropy.
 
In the commutative diagram 
 \begin{equation*}
\begin{xy}
\xymatrix{
X^H\ar@{-->}[dr]^(0.8){i}\ar[ddr]_{f^H}\ar[rrd]^{\iota^H_X}&&\\
&X\times_YY^H\ar[r]_j\ar[d]^{f'}&X\ar[d]^f\\
&Y^H\ar[r]^{\iota^H_Y}&Y
}
\end{xy}
\end{equation*}
the morphisms $\iota^H_X$ and $\iota^H_Y$ are closed immersions, hence so are $j$ and $i$. From the isotropy condition (*) it follows that $f^H$ is surjective, so that by dimension $X^H$ is a union of irreducible components of $X\times_YY^H$ and thus $i$ and also $f^H$ are étale. If for any $y\in Y^H$ an element $x\in X$ is given with the property that $f$ induces isomorphisms of the respective residue fields and scheme theoretic stabilizers, then $x$ is in $X^H$ and therefore $f^H$ is Nisnevich.

Conversely, let $f^H$ be a Nisnevich cover in $\smk$ for all subgroups $H\leq G$. Given an element $y\in Y$ say with $G_y=H\times \kappa(y)$, then $y$ is in $Y^H$ and there is an element $x$ in $X^H$, such that $f$ induces an isomorphism from $\kappa(y)$ to $\kappa(x)$. Since $x$ is in $X^H$ we know
$$
G_x=K\times\kappa(x)\geq H\times \kappa(x)\cong H\times \kappa(y) = G_y
$$
and the equivariance of $f$ implies $G_x\leq G_y$, so that $f$ induces an isomorphism on scheme theoretic isotropy.
\end{proof}

\begin{corollary}\label{cor:eqNisisHNis}
 Every equivariant Nisnevich cover is an $H$-cover.
\end{corollary}
\begin{proof}
 This follows from the above lemma combined with Lemma \ref{lemma:setschemestabilizer}.
\end{proof}

The following example reminds one to be careful while thinking about isotropy groups and fixed points.

\begin{bsp}\label{ex:fixedpointsofgalois}
Let $L:k$ be a finite Galois extension and $G=\mathrm{Gal}(L:k)$. The induced $G$-action on $\spec(L)$ has empty fixed points $\spec(L)^G=\emptyset$. This is since $\spec(L)^G$ is by construction a closed subscheme of $\spec(L)$ and 
$$
\Hom_{\smk}(\spec(L),\spec(L)^G)\cong\Hom_{\Gsmk}(\spec(L)_{\mathrm{tr}},\spec(L))=\emptyset.
$$
 
The set-theoretic stabilizer $S_\ast$ of the unique point $\ast$ is obviously the whole group $G$, but the scheme theoretic stabilizer is trivial, that is $G_\ast=\spec(L)$, since the action is free and hence the left vertical arrow in the pullback diagram
\begin{equation*}
\begin{xy}
\xymatrix{
G_\ast\ar[r]\ar[d]_\cong& G\times \spec(L)\ar[d]^\cong_\Psi\\
\spec(L)\ar[r]^(0.4)\Delta&\spec(L)\times\spec(L)
}
\end{xy}
\end{equation*}
is an isomorphism as well.
\end{bsp}

\begin{lemma}
 The $H$-Nisnevich topology is subcanonical, i.e.~representable pre\-sheaves are sheaves on $\Gsmk$.
\end{lemma}
\begin{proof}
 Let $\{Z_i\xrightarrow{\iota_i} Z\}_i$ be an $H$-Nisnevich covering and let $U:\Gsmk\to\smk$ be the forgetful functor. $U$ is faithful and as a (trivial) fixed point functor $U$ takes the chosen covering to a Nisnevich covering in $\smk$. Hence, the bottom row in the diagram
\begin{equation*}
\begin{xy}
\xymatrix{
\Hom_G(Z,X)\ar[r]\ar[d]&\prod\Hom_G(Z_i,X)\ar[d]\ar@<2pt>[r]\ar@<-1pt>[r]&\prod\Hom_G(Z_i\times_ZZ_j,X)\ar[d]\\
\Hom_k(UZ,UX)\ar[r]&\prod\Hom_k(UZ_i,UX)\ar@<2pt>[r]\ar@<-1pt>[r]&\prod\Hom_k(UZ_i\times_{UZ}UZ_j),UX)
}
\end{xy}
\end{equation*}
is an equalizer and all vertical arrows are injective. A family $(\sigma_i)_i$ in  the product $\prod\Hom_G(Z_i,X)$ which is equalized  by the double arrow is mapped to a family in $\prod\Hom_k(UZ_i,UX)$ which is also equalized and therefore comes from a morphism $g$ in $\Hom_k(UZ,UX)$. To see that $g$ is equivariant we have to show that the square labeled with '?' commutes in the following diagram. 
\begin{equation*}
\begin{xy}
\xymatrix{
G\times\coprod Z_i\ar[r]^{\id_G\times\coprod f_i}\ar[d] &G\times Z \ar[d]^{\alpha_Z}\ar[r]^{\id_G\times g}&G\times X\ar[d]^{\alpha_X}\\
\coprod Z_i\ar[r]^{\coprod\iota_i}\ar@<-5pt>@{}[ur]^{\circlearrowleft}& Z\ar[r]_g\ar@{}@<-5pt>[ur]^{?} & X
}
\end{xy}
\end{equation*}
First note that all $\iota_i$ and $g\circ\iota_i$ are equivariant. The square in question commutes since both the outer rectangle and the left square commute, and since $\id_G\times\coprod f_i$ is an epimorphism.
\end{proof}

\begin{corollary}
 The equivariant Nisnevich topology is also subcanonical.
\end{corollary}

\begin{lemma}\label{lemma:continousmap}
 For all $H\leq G$, the $H$ fixed points functor $(-)^H:\Gsmk\to\smk$ is continuous map of sites.
\end{lemma}
\begin{proof}
 \cite[III.Proposition 1.6.]{SGA4}
\end{proof}

\begin{lemma}\label{lemma:adjunction}
The adjunction \eqref{eq:schemeadj} extends via left Kan extension of $(-)^G$ to an adjunction
\begin{equation}\label{eq:adjunction}
 ((-)^G)_*:\sShv(\Gsmk)\rightleftarrows\sShv(\smk):R^G,
\end{equation}
where the right adjoint is composition with $(-)^G$.
\end{lemma}
\begin{proof}
 Consider the situation
\begin{equation*}
\begin{xy}
\xymatrix@R=10pt{
\Gsmk\ar[rr]^Y\ar[rd]_{(-)^G}&&\sPre(\Gsmk)\ar@{-->}[dd]_{L}\\
&\smk\ar[rd]_Y&\\
&&\sPre(\smk)\ar@<-1ex>@{-->}[uu]_{R}
}
\end{xy}
\end{equation*}
where $L$ is the left Kan extension of $Y\circ(-)^G$ along the horizontal Yoneda embedding $Y$ and $R$ is the right adjoint of $L$. The right adjoint $R$ is given by composition with $(-)^G$, which is a continuous map of sites and so $R$ restricts to a functor $R'$ in 
\begin{equation*}
\begin{xy}
\xymatrix@R=20pt{
\sPre(\Gsmk)\ar@<0.5ex>[r]^{a_1}\ar@<-1ex>[d]_{L}&\sShv(\Gsmk)\ar@<0.5ex>[l]^{i_1}\\
\sPre(\smk)\ar@<0.5ex>[r]^{a_2}\ar[u]_R&\sShv(\smk)\ar[u]_{R'}\ar@<0.5ex>[l]^{i_2}
}
\end{xy}
\end{equation*}
of sheaves with respect to the Nisnevich (resp.~$H$-Nisnevich) topology. Thus, we have that $((-)^G)_*:=a_2Li_1$ is right adjoint to $R^G$.
\end{proof}

From now on we will mostly leave sheaves aside and focus on a theory of presheaves. The few statements about sheaves we collected so far were just given to allow a study of points for this $H$-Nisnevich topology on $\Gsmk$ in the next subsection.

For any subgroup $H\leq G$, we define the $H$-fixed points functor 
\begin{equation}\label{eq:Hfixed}
(-)^H:\sPre(\Gsmk)\to \sPre(\smk)
\end{equation}
as the composite
\begin{equation*}
\begin{xy}
\xymatrix@R=0cm{
\sPre(\Gsmk)\ar[r]^{\res_H}&\sPre(H\smk)\ar[r]^{(-)^H}&\sPre(\smk)\\
X\ar@{|->}[r]&X(G\times_H-),&
}
\end{xy}
\end{equation*}
where $\res_H$ is the restriction functor or forgetful functor. Note that precomposing with the induction functor $G\times_H-$ coincides with the left Kan extension of $\res_H:\Gsmk\to H\smk$. Hence we could have equivalently defined $H$-fixed points as a left Kan extension in one step.

\begin{rem}\label{rem:fixedisright}
The functor $(-)^G:\sPre(\Gsmk)\to\sPre(\smk)$ is also right adjoint which can be seen as follows. On the scheme level we have the adjunction $(-)\tr:\smk\rightleftarrows\Gsmk:(-)^G$ with the left adjoint given by the trivial $G$-action functor $(-)\tr$. The right adjoint $R$ to the left Kan extension of $(-)\tr$ along the obvious Yoneda embedding is given by precomposition with $(-)\tr$ and hence commutes with colimits. Further, for a representable sheaf $\widetilde{X}$ we evaluate
$$
R(\widetilde{X})(U)\cong\Hom_{\Gsmk}(U\tr,X)\cong\Hom_{\smk}(U,X^G)=\widetilde{X^G}(U)
$$
and note that $R$ and $(-)^G$ coincide on representables and therefore are equal. The same arguments work to show that $\res_H:\sPre(\Gsmk)\to\sPre(H\smk)$ is also right adjoint and we eventually note that the $H$-fixed points functor $(-)^H:\sPre(\Gsmk)\to\sPre(\smk)$ from \eqref{eq:Hfixed} is a left and right adjoint functor, for all $H\leq G$.
\end{rem}

\section{Equivariant Motivic Homotopy Theory}
The following example was explained to me by Ben Williams. It shows that local weak equivalences with respect to the equivariant Nisnevich topology can not be detected by the family $\{(-)^H\}_{H\leq G}$ of fixed-point functors.

\begin{bsp}
 Let $Y$ in $\ZZ/2\text-\sm{\CC}$ be given by the disjoint union $\GG_m\coprod\GG_m$ be equipped with the $\ZZ/2$ action permuting the summands. Let $X=\GG_m=\spec(\CC[T,T^{-1}])$ carry the $\ZZ/2$ action induced by $T\mapsto -T$.
 We define a $\ZZ/2$-equivariant morphism
 $$
  p:=\id\coprod\sigma:Y\to X,
 $$
 where $\sigma$ is the non-trivial automorphism acting on $X$. Note that the fixed-point morphisms 
\begin{align*}
p^e&={\GG_m}\coprod{\GG_m}\xrightarrow{\id\coprod\id}\GG_m \text{ and}\\
p^G&=\id_\emptyset  
\end{align*}
are Nisnevich covers in the usual non-equivariant sense. Now, consider the coequalizer diagram
\begin{equation*}
\begin{xy}
\xymatrix{
Y\times_XY\ar@<0.5ex>[r]\ar@<-0.5ex>[r]&Y\ar[r]\ar^p[dr]&W\ar@{-->}^h[d]\\
&&X
}
\end{xy}
\end{equation*}
The map $h$ is not a local weak equivalence in the equivariant Nisnevich topology and $p$ is not a cover in that topology. The reason is that the generic point of $X=\GG_m$ inherits an action and does not lift to $Y$: There is a map
$$
(t\mapsto -t) \circlearrowright \spec(\CC(t))\to X=\GG_m,
$$
but the value of the point at $Y$ and $W$ is $\emptyset$ since
$$
\emptyset=\Hom_G(\CC(t),Y)\twoheadrightarrow\Hom_G(\CC(t),W).
$$
Hence, $h$ is not a local weak equivalence for the equivariant Nisnevich topology and the morphism $p$ can not be a covering for this topology.
\end{bsp}

\subsection{Characterization of Unstable Equivalences}
\label{sec:charofunstableweq}

Recall that a point $x$ in a topos $T$ is a geometric morphism $x:\set\to T$ or equivalently, by Freyd's Theorem, a functor $x^*:T\to\set$ which commutes with colimits and finite limits. In this subsection $\Gsmk$ is equipped with the $H$-Nisnevich topology by default.

Denote by $\mathcal{H}\mathrm{ensel}:=\{x^*:F\mapsto F(\spec(\mathcal O^{h}_{X,x})\mid x\in X\}_X$ the set of functors indexed over all $X$ in a small skeleton of $\smk$. This gives a conservative family of points for the Nisnevich topology on $\smk$ \cite[Lemma 3.1.11]{mv99}, i.e. a morphism $f$ in $\sPre(\smk)$ is local weak equivalence if and only if $x^*f$ is a weak equivalence of simplicial sets for all $x,X$.
 
\begin{lemma}\label{lemma:fixedarepoints}
Let $x^*$ be a point in $\sShv(\smk)$. Then the composition $x^*\circ (-)^H$ is a point in $\sShv(\Gsmk)$. Hence, if $f\in\sPre(\Gsmk)$ is a local weak equivalence, then $f^H$ is a local weak equivalence in $\sPre(\smk)$. 
\end{lemma}
\begin{proof}
By Remark \ref{rem:fixedisright} the left Kan extension $$(-)^H:\sPre(\Gsmk)\to\sPre(\smk)$$ is also a right adjoint and therefore preserves limits. As a left adjoint it preserves colimits and hence $x^*\circ(-)^H$ is a point in $\sShv(\Gsmk)$. Thus, for any local weak equivalence $f\in\sPre(\Gsmk)$ the morphism $x^*f^H$ is weak equivalence of simplicial sets, so $f^H$ is local weak equivalence in $\sPre(\smk)$.
\end{proof}

\begin{lemma}\label{lemma:fixedconservative}
 The set of functors $\sShv(\Gsmk)\to\set$ given by
$$
\left\lbrace x^*\circ(-)^H \mid H\leq G, x^*\in\mathcal{H}\mathrm{ensel}  \right\rbrace 
$$
is a conservative family of points in $\sShv(\Gsmk)$ (for the $H$-Nisnevich topology).
\end{lemma}
\begin{proof}
Let $\mathfrak{X}:=(f_j^H:X_j\to X)_{j\in J}$ be a family of morphisms in $\Gsmk$ such that 
$$
\left(   x^*(X_j^H\xrightarrow{f_j^H}X^H) \right)_{j\in J} 
$$
is surjective for all Nisnevich points $x^*\in\mathcal{H}\mathrm{ensel}$ and $H\leq G$. 
Then by \cite[Proposition 6.5.a]{SGA4}, $(  f_j^H: X_j^H\rightarrow X^H )_{j\in J} 
$ is a Nisnevich covering in $\smk$. Hence, $\mathfrak{X}$ is a $H$-Nisnevich covering.
\end{proof}

The following is also an immediate consequence.

\begin{corollary}\label{cor:charofeqlocaleq}
A morphism $f\in \sPre(\Gsmk)$ is a local weak equivalence if and only if for all subgroups $H\leq G$ the morphism $f^H$ is a local weak equivalence in $\sPre(\smk)$.
\end{corollary}

\begin{corollary}\label{cor:Hlocalleftquillen}
For all subgroups $H\leq G$, the adjunction 
$$
(-)^H:\sPre(\Gsmk)\rightleftarrows\sPre(\smk):R_H
$$
is a Quillen adjunction for the local injective model structures.
\end{corollary}
\begin{proof}
We have just concluded that $(-)^H$ preserves local weak equivalences. Because of being right adjoint (and the fact that both categories have pullbacks) the functor $(-)^H$ also preserves monomorphisms, i.e. local injective cofibrations.
\end{proof}

To achieve the same result for $\AA^1$-local weak equivalences we cite a result of Hirschhorn which takes care of the Bousfield localization on both sides of a Quillen adjunction.

\begin{proposition}\label{prop:hirschhorn} Let $F:\C\rightleftarrows\D:G$ be a Quillen pair and let $K$ be a class of morphisms in $\C$. Denote by $L_K\C$, resp.~$L_{\LL_FK}\D$, the left Bousfield localization of $\C$ with respect to $K$, resp.~of $\D$ with respect to the image of $K$ under the left derived of $F$. Then $F:L_K\C\rightleftarrows L_{\LL_FK}\D :G$ remains a Quillen pair. 
\end{proposition}
\begin{proof}
 \cite[Theorem 3.3.20]{hirsch}
\end{proof}

\begin{lemma}\label{lemma:HlocalrightQuillen}
 Let $H,K\leq G$. The composition $(-)^K\circ (-)_H:\sPre(\smk)\to\sPre(\smk)$ equals some coproduct of identities. In particular, the $H$-fixed points functors $(-)^H$ are right Quillen functors in a Quillen adjunction
 $$
 (-)_H:\sPre(\smk)\leftrightarrows\sPre(\Gsmk):(-)^H
 $$
 with respect to the local injective model structures.
\end{lemma}
\begin{proof}
 Both functors commute with colimits, so we only need to check the statement for representables. We have
 $$
 ((\widetilde{Y})_H)^K \cong \left(\widetilde{G/H\times Y}\right)^K \cong \widetilde{(G/H)^K\times Y} \cong \coprod_{(G/H)^K}\widetilde Y.
 $$
 Furthermore, the functors $(-)^K$ detect local weak equivalences by Corollary \ref{cor:charofeqlocaleq} and a (finite) coproduct of local weak equivalences is a local weak equivalence. Eventually, to check that $(-)_H$ preserves monomorphisms recall that $(-)_H$ is the left Yoneda extension of $G/H\times-:\smk\to\Gsmk$ which preserves all finite limits. Left Kan extensions of flat functors preserve finite limits and in particular monomorphisms.
\end{proof}

\begin{lemma}\label{lemma:HA1rightQuillen}
 For every subgroup $H\leq G$, the $H$-fixed points functor $(-)^H$ is a right Quillen functor in the adjunction
 $$
 (-)_H:\sPre(\smk)\leftrightarrows\sPre(\Gsmk):(-)^H
 $$
 with respect to the $\AA^1$-local injective model structures.
\end{lemma}
\begin{proof}
 By Proposition \ref{prop:hirschhorn} the Quillen adjunction 
 $$
 (-)_H:\sPre(\smk)\leftrightarrows\sPre(\Gsmk):(-)^H
 $$
 of Lemma \ref{lemma:HlocalrightQuillen} descents to a Quillen adjunction 
 \begin{equation*}
\begin{xy}
\xymatrix{
L_K\sPre(\smk)\ar@<0.5ex>[r]^{(-)_H}&L_{\mathbb L_{(-)_H}K}\sPre(\Gsmk)\ar@<0.5ex>[l]^{(-)^H}
}
\end{xy}
\end{equation*} 
of left Bousfield localizations, where $K$ is the class of morphisms represented by $\{X\times\AA^1\to X\mid X\in \smk\}$ and $\mathbb L_{(-)_H}K$ is the image of that class under the total left derived of $(-)_H$. The latter is a (proper) subclass of the class of morphisms represented by $\{X\times\AA^1\to X\mid X\in \Gsmk\}$ which is used to $\AA^1$-localize on the equivariant side. Hence, the identity gives a left Quillen functor
$$
L_{\mathbb L_{(-)_H}K}\sPre(\Gsmk)\to \sPre(\Gsmk)
$$
where the right hand side carries the $\AA^1$-local injective model structure. By composing the two Quillen adjunctions we obtain the conclusion.
\end{proof}

\begin{proposition}\label{prop:charofA1localequivariantweq}
A morphism $f\in \sPre(\Gsmk)$ is an $\AA^1$-local weak equivalence if and only if for all subgroups $H\leq G$ the morphism $f^H$ is an $\AA^1$-local weak equivalence in $\sPre(\smk)$.
\end{proposition}
\begin{proof}
By Proposition \ref{prop:hirschhorn} the functors $(-)^H$ are left Quillen functors for the $\AA^1$-local injective model structures. Thus, it follows by Ken Brown's Lemma \cite[Lemma 7.7.1]{hirsch} that $(-)^H$ \emph{preserves} $\AA^1$-local weak equivalences.

Conversely, suppose that $f:X\to Y$ in $\sPre(\Gsmk)$ is a map such that for all subgroups $H$ of $G$, the morphism $f^H\in\sPre(\smk)$ is an $\AA^1$-local weak equivalence. Let $r$ be a fibrant replacement functor in the $\AA^1$-local injective structure on $\sPre(\Gsmk)$. Then $(-)^H$ takes the diagram
\begin{equation*}
\begin{xy}
\xymatrix{
X\ar[d]_{\sim_{\AA^1}}\ar[r]^f&Y\ar[d]^{\sim_{\AA^1}}\\
rX\ar[r]^{rf}&rY
}
\end{xy}
\end{equation*}
to the diagram
\begin{equation*}
\begin{xy}
\xymatrix{
X^H\ar[d]_{\sim_{\AA^1}}\ar[r]^{f^H}_{\sim_{\AA^1}}&Y^H\ar[d]^{\sim_{\AA^1}}\\
(rX)^H\ar[r]^{(rf)^H}&(rY)^H
}
\end{xy}
\end{equation*}
where all the arrows decorated with $\sim_{\AA^1}$ are $\AA^1$-local weak equivalences. Hence $(rf)^H$ is an $\AA^1$-local weak equivalence between objects which are $\AA^1$-locally injective fibrant by Lemma \ref{lemma:HA1rightQuillen}. Therefore, $(rf)^H$ is a local weak equivalence for all $H$ and it follows by Corollary \ref{cor:charofeqlocaleq} that $rf$ is a local weak equivalence and so $f$ is an $\AA^1$-local weak equivalence.
\end{proof}

\section{Stable Equivariant Motivic Homotopy Theory}

\subsection{The Stable Model Category}
The definition of representation spheres below already aims towards a stable equivariant homotopy theory. Analogously to the work of Mandell \cite{mandell2004equivariant} in classical topology, and to Hu, Kriz, and Ormsby in \cite{Hu2011homotopy} we consider  spectra with respect to smashing with the regular representation sphere.

\begin{defi}
Let $V\in\Gsmk$ be a representation of $G$. We define the representation sphere $S^V$ to be the quotient
$$
V/(V-0)
$$ 
in $\sPre(\Gsmk)$. For the special case of the regular representation we introduce the notation $$\TT_G:=S^{\AA[G]}.$$
\end{defi}

\begin{rem} A linear algebraic group is called linearly reductive if every rational representation is completely reducible. It is the statement of Maschke's Theorem that a finite group is linearly reductive if the characteristic of $k$ does not divide the group order. A splitting of the representation $V$ causes a splitting of the representation sphere:
$$
S^{V\oplus W} \cong S^V\smash S^W.
$$
Clearly, the reason to invert the regular representation sphere is to invert smashing with all representation spheres and therefore it should be emphasized that the group $G$ has to be linearly reductive for this approach to make sense.

However, there are models for stable homotopy theory based on enriched functors \cite{lydakis, blumberg2005continuous, enrichedfun} instead of sequential spectra. These allow a more flexible stabilization and in a recent preprint \cite{carlsson2011equivariant} Carlsson and Joshua apply this technique to stabilize a slightly different approach to equivariant motivic homotopy theory without being restricted to linearly reductive groups.
\end{rem}

The category $\Nspec(\C,Q)$ of sequential spectra in a model category $\C$ with respect to a left Quillen functor $Q:\C\to\C$ consists of objects
$$
(X_n,\sigma_n)_{n\in\NN},
$$
where the $X_n$'s are objects in $\C$ and $\sigma_n:Q(X_n)\to X_{n+1}$ are morphisms in $\C$, the so-called bonding maps. The morphisms in $\Nspec(\C,Q)$ are given by sequences of morphisms in $\C$ which commute with the respective bonding maps.

There is the usual Yoga of model structures for stable homotopy theory in the sense of spectra in general model categories (cf.~\cite{hovey2001spectra}) that also applies to the equivariant and non-equivariant stable motivic homotopy theory as developed below.  We depict our procedure in the following diagram, where in the top row the relevant categories of equivariant motivic spaces, sequential and symmetric spectra and their standard Quillen adjunctions show up. Below the top row, various model structures appear and are connected by arrows.
\begin{equation}\label{eq:stableyoga}
\begin{xy}
\xymatrix@R=8pt{
\sPre_.(\Gsmk)\ar@<0.5ex>[r]^(0.4){\Sigma^\infty} & \Nspec(\sPre_.(\Gsmk),\TT_G\smash -)\ar@<0.5ex>[l]^(0.6){\Omega^\infty}\\
\text{(1) local injective}\ar[d] & \text{(3) levelwise} \ar[d]\\
\text{(2) } \AA^1\text{-local injective} \ar[ur]& \text{(4) stable} 
}
\end{xy}
\end{equation}
Here, we choose to start with the local injective model structure (1) on pointed simplicial presheaves, in which the cofibrations are given by monomorphisms and weak equivalences are the local weak equivalences after forgetting the basepoint. The vertical arrows mean Bousfield localization, in this case at the class
$$
\left\lbrace X\smash\AA^1_+\to X \mid X\in\Gsmk   \right\rbrace
$$
which gives the $\AA^1$-local injective model structure (2). This model structure can be lifted to a projective levelwise model structure on sequential $\TT_G$-spectra \cite[Lemma 2.1]{MSS} (3), which can be localized at the class of stable equivalences to result in a stable model structure (4). 

Fortunately, compared with Hovey's general setup, we are in the good situation of \cite[Theorem 4.9]{hovey2001spectra} and thus we may proceed as Jardine in \cite{MSS} to define stable weak equivalences.

\begin{lemma}
 The adjunction 
 $$
 \TT_G\smash - :\sPre_.(\Gsmk)\rightleftarrows\sPre_.(\Gsmk):\Omega_{\TT_G}
 $$
 prolongates canonically to an adjunction
  $$
 \Sigma'_{\TT_G}:\Nspec(\sPre_.(\Gsmk),\TT_G\smash-)\rightleftarrows\Nspec(\sPre_.(\Gsmk),\TT_G\smash-):\Omega'_{\TT_G}
 $$
 called \emph{fake suspension adjunction}.
\end{lemma}
\begin{proof}
 Use the identity transformation on $(\TT_G\smash-)^2$ to prolongate $\TT_G\smash- $ and compose unit and counit of the adjunction to obtain a natural transformation 
 $$
 \TT_G\smash(\Omega_{\TT_G}(-))\to \Omega_{\TT_G}(\TT_G\smash -))
 $$
 which prolongates $\Omega_{\TT_G}$ to the right adjoint.
\end{proof}

\begin{rem}
 The above lemma is originally \cite[Corollary 1.6]{hovey2001spectra} in the general situation. Note that there is no twisting of the smash factors involved in the bonding maps, which is why the resulting suspension is called fake suspension in contrast to the suspension defined in \eqref{eq:defsuspension}.
\end{rem}

\begin{defi}\label{def:stableweq}
Let $R$ denote a levelwise fibrant replacement functor. A morphism $f\in \Nspec(\sPre_.(\Gsmk),\TT_G\smash-)$ is called a stable equivalence if 
$$
(\Omega'\circ\mathrm{sh})^\infty R(f)
$$
is a levelwise equivalence.  
\end{defi}

For Jardine's machinery to work, we need to assure that the object $\TT_G$ which is used for suspending fulfills a technical property, which then implies a good behavior of the right adjoint to smashing with $\TT_G$.

\begin{lemma}
 The object $\TT_G\in\sPre_.(\Gsmk)$ is compact in the sense of \cite[2.2]{MSS}.
\end{lemma}
\begin{proof} The analog statement about the presheaf quotient $\AA^1/(\AA^1\setminus0)$ in Jardine's work is \cite[Lemma 2.2]{MSS}. All the arguments in the proof are statements about the flasque model structure on simplicial presheaves on a general site \cite{flasque}. The only thing used about about schemes is that an inclusion of schemes gives a monomorphism of the represented presheaves, which is true for an inclusion of equivariant schemes like $(\AA[G]\setminus0)\hookrightarrow\AA[G]$ as well.
\end{proof}

\begin{theorem}\label{thm:stableeqms}
 Let $T$ be a compact object in $\sPre_.(\Gsmk)$. There is a proper simplicial model structure on the associated category $\Nspec(\sPre_.(\Gsmk),T\smash -)$ of $T$-spectra with stable weak equivalences and stable fibrations.
\end{theorem}
\begin{proof}
 This works as in \cite[Theorem 2.9]{MSS}.
\end{proof}

\begin{defi}
 Let $X$ in $\Nspec(\sPre_.(\Gsmk),\TT_G\smash-)$. We define the suspension $\Sigma_{\TT_G}X$ by $\Sigma_{\TT_G}X_n=\TT_G\smash X_n$ with bonding maps
 $$
 \sigma_{\Sigma X}:\TT_G\smash \TT_G\smash X_n\xrightarrow{\tau\smash\id_{X_n}}\TT_G\smash \TT_G\smash X_n\xrightarrow{\sigma_X}\TT_G\smash X_{n+1}
 $$
 where $\tau:\TT_G\smash \TT_G\to\TT_G\smash \TT_G$ denotes the twist of the two smash factors. The right adjoint to $\Sigma_{\TT_G}$ is also levelwise given by the internal hom $\Omega_{\TT_G}$, i.e.~$\Omega_{\TT_G}(X)_n=\Omega_{\TT_G}(X_n)$ with bonding maps adjoint to 
 $$
 X_n\smash \TT_G\xrightarrow{\tau}\TT_G\smash X_n\xrightarrow{\sigma_X}X_{n+1}.
 $$
 Together these two functors give the suspension adjunction
\begin{equation}\label{eq:defsuspension}
\Sigma_{\TT_G}:\Nspec(\sPre_.(\Gsmk),\TT_G\smash-)\rightleftarrows\Nspec(\sPre_.(\Gsmk),\TT_G\smash-):\Omega_{\TT_G}.
\end{equation}
\end{defi}

To be able to untwist the levelwise smashing inside the definition of the functor $\TT_G\smash -$ an important condition appears to be the symmetry of $\TT_G$.

\begin{lemma}
 There is an $\AA^1$-homotopy in $\psPre(\Gsmk)$ between the cyclic permutation of the smash factors
 $$
 \TT_G\smash\TT_G\smash\TT_G\to\TT_G\smash\TT_G\smash\TT_G
 $$
 and the identity.
\end{lemma}
\begin{proof}
This is \cite[Lemma 2]{Hu2011homotopy} for the $\AA^1$-local model structure with respect to the equivariant Nisnevich topology, but the topology on $\Gsmk$ does not matter for this statement to hold.
\end{proof}

A consequence, which is also true in the more general situation of Hovey's \cite[Theorem 9.3]{hovey2001spectra}, is that smashing with $\TT_G$ is invertible in the stable model.

\begin{theorem}
The suspension adjunction \eqref{eq:defsuspension} is a Quillen equivalence with respect to the stable model structure.
\end{theorem}
\begin{proof}
 Let $Y$ be fibrant and $f:\TT_G\smash X\to Y$ in $\Nspec(\sPre_.(\Gsmk), \TT_G\smash -)$. By \cite[Corollary 3.16]{MSS}
 $$
 ev:\TT_G\smash\Omega_{\TT_G}Y\to Y
 $$
 is a stable equivalence, so we may deduce from the commutative diagram
 \begin{equation*}   
\begin{xy}
\xymatrix{
&\TT_G\smash\Omega_{\TT_G}Y\ar[d]^{ev}_\sim\\
\TT_G\smash X\ar[ur]^{\TT f^\sharp}\ar[r]^f&Y
}
\end{xy}
\end{equation*}
that $f$ is a stable equivalence if and only if $\TT f^\sharp$ is a stable equivalence, which is by \cite[Corollary 3.18]{MSS} if and only if the adjoint morphism $f^\sharp$ is a stable equivalence.
\end{proof}

\begin{proposition}\label{prop:stabilizationcontrol}  Let $V$ be a representation of $G$. Then the adjunction
$$
-\smash S^V:\Nspec(\sPre_.(\Gsmk),\TT_G\smash-) \rightleftarrows\Nspec(\sPre_.(\Gsmk),\TT_G\smash-):\Omega^V
$$
is a Quillen equivalence.
\end{proposition}
\begin{proof}
 Smashing with $S^V$ is a left Quillen functor. There exists a representation $W$ such that $V\oplus W \cong n\AA^G$ is a $n$-fold sum of the regular representation. By using the theorem above one can show that $\Omega^{n\TT_G}\circ S^W$ is 'Quillen inverse' to $S^V$.
\end{proof}

In Definition \ref{def:stableweq} a morphism $f:X\to Y$ of equivariant spectra was defined to be a stable equivalence if $\colim_i (\Omega'\circ\mathrm{sh})^i R(f)$ is a levelwise equivalence of equivariant spectra. Equivalently, for all $m,n\in \NN$ and all $H\leq G$ the induced maps of all sectionwise $n$-th homotopy groups in level $m$ of the $H$-fixed points are isomorphisms, i.e.
\begin{equation}\label{eq:leadstostablepi}
f_*:\colim_i [G/H\smash S^n\smash \TT_G^i,X_{m+i}|_U]\to \colim_i [G/H\smash S^n\smash \TT_G^i,Y_{m+i}|_U] 
\end{equation}
is an isomorphism of groups for all $U\in \smk$.

The standard simplicial enrichment of local homotopy theory on $\sPre(\C)$ gives us another splitting of $\TT_G$.
\begin{lemma}\label{lemma:TT_Gsplitting1}
 There is an isomorphism $\TT_G\cong S^1\smash (\AA[G]-0)$ in the unstable equivariant homotopy category.
\end{lemma}
\begin{proof} Recall that $\TT_G\cong \AA[G]/(\AA[G]-0)$ where $\AA[G]$ is pointed by $1$ and consider the diagram 
 \begin{equation*}
\begin{xy}
\xymatrix{
\partial\Delta[1]\smash(\AA[G]-0)\ar@{(->}[r]_{}\ar@{(->}[d]&\AA[G]\ar[d]\ar@{->}[r]^{\sim}&\ast\ar[d]\\
\Delta[1]\smash (\AA[G]-0)\ar[r]\ar[d]_{\sim}&P\ar[d]\ar[r]&S^1\smash (\AA[G]-0)\\
\ast\ar[r]&\TT_G
}
\end{xy}
\end{equation*}
consisting of push out squares. The two morphisms decorated with a tilde are $\AA^1$-local weak equivalences. The vertical one being 
$$
\Delta[1]\smash(\AA[G]-0)\xrightarrow{p\smash \id}\Delta[0]\smash (\AA[G]-0)=\ast 
$$
and the horizontal one by Proposition \ref{prop:charofA1localequivariantweq}. Further, both morphisms to the push out $P$ are cofibrations and hence by left properness there is a zig-zag
$$
\TT_G\xleftarrow{\sim}P\xrightarrow{\sim}S^1\smash(\AA[G]-0)
$$
of weak equivalences.
\end{proof}

Continuing from \eqref{eq:leadstostablepi} we compute that $f$ is a stable equivalence if and only if the induced map
$$
\colim_i [G/H\smash S^{n+i}\smash (\AA[G]-0)^i,X_{m+i}|_U]\to \colim_i [G/H\smash S^{n+i}\smash (\AA[G]-0)^i,Y_{m+i}|_U] 
$$
is an isomorphism. This leads naturally to the following definition.

\begin{defi}
 Let $X$ in $\Nspec(\sPre_.(\Gsmk),\TT_G\smash -)$. The \emph{weighted stable homotopy groups} $\pi^H_{s,t}X$ are defined to be the presheaf of groups on $\smk$ given by
 $$
 \pi_{s,t}^H(X)(U)=\colim_{i\geq0} [G/H\smash S^{s+i}\smash (\AA[G]-0)^{t+i}\smash U_+,X_i]
 $$
\end{defi}

\begin{lemma}
A morphism $f:X\to Y$ of equivariant spectra is a stable equivalence if and only if it induces isomorphisms
$$
\pi^H_{s,t}(f):\pi^H_{s,t}(X)\xrightarrow{\cong}\pi^H_{s,t}(Y)
$$
for all $s,t\in\ZZ$ and $H\leq G$.
\end{lemma}
\begin{proof}
 This is the analog of \cite[Lemma 3.7]{MSS}.
\end{proof}

\subsubsection*{Cofiber and Fiber Sequences}
Recall from Theorem \ref{thm:stableeqms} and Proposition \ref{prop:stabilizationcontrol} that we consider $\Nspec(\Gsmk)$ as a proper stable model category. The theory of cofiber and fiber sequences is therefore quite convenient.
 Given a morphism $f:X\to Y$ of equivariant spectra the homotopy cofiber (resp.~homotopy fiber) is defined by the homotopy push out (resp.~homotopy pullback) square
\begin{equation*}
 \begin{xy}
\xymatrix{
X\ar[r]^f\ar[d]&Y\ar[d]& &hofib(f)\ar[r]\ar[d]&\ast\ar[d]\\
\ast\ar[r]& hocofib(f)& & X\ar[r]^f & Y.
}
\end{xy}
\end{equation*}
The simplicial structure on $\Nspec(\Gsmk)$ provided by Theorem \ref{thm:stableeqms} implies that there is a stable weak equivalence
$$
hocofib(X\to\ast)\simeq S^1\smash X.
$$

At this point we omit a thorough introduction of the triangulated structure on the stable homotopy category $\SHKG$ via $S^1\mid(\AA[G]-0)$-bispectra and (co-) fiber sequences which works out perfectly analogous to what is developed in Jardine's Section 3.3 of \cite{MSS}. Instead, we just state the following important consequence.

\begin{lemma}
 Given a cofiber sequence 
 $$
 X\xrightarrow{f}Y\to hocofib(f)
 $$
 of equivariant spectra, there is a long exact sequence of presheaves of groups
\begin{equation}\label{eq:lesofhomotopygroups}
 \ldots\to\pi^G_{s,t}(X)\to\pi^G_{s,t}(Y)\to\pi^G_{s,t}(hocofib(f))\to\pi^G_{s-1,t}(X)\to\ldots  
 \end{equation}
\end{lemma}

\subsection{Naive $G$-Spectra and Change of Universe} For a smooth connection between stable equivariant and non-equivariant homotopy theories it is convenient to introduce naive $G$-spectra, a natural intermediate. We mirror some results from the topological theory (cf.~\cite[Section II]{lewis1986equivariant}).

\begin{defi} An object in $\Nspec(\sPre_.(\Gsmk),T\smash -)$ is called a (sequential) \emph{naive} $G$-spectrum. We consider the category $\Nspec(\sPre_.(\Gsmk),T\smash -)$ of naive $G$-spectra as endowed with the stable model structure analogous to \eqref{eq:stableyoga}, i.e.~take the $\AA^1$-local injective model structure with respect to the $H$-Nisnevich topology on $\sPre_.(\Gsmk)$ and localize the levelwise (projective) model structure on $\Nspec(\sPre_.(\Gsmk),T\smash-)$ along stable equivalences.
\end{defi}

We will usually continue to call an object $E$ in $\Nspec(\sPre_.(\Gsmk),\TT_G\smash-)$ an equivariant spectrum or $G$-spectrum, but to emphasize the distinction $E$ is sometimes called a genuine $G$-spectrum.

Given a non-equivariant spectrum $X$ in $\Nspec(\sPre_.(\smk))$ we may apply the canonical prolongation of the trivial $G$-action functor \eqref{eq:trivialGactionKANextended}
$$
(-)\tr:\sPre(\smk)\to\sPre(\Gsmk)
$$
on $X$ to obtain a naive $G$-spectrum $X\tr$. Let $E$ be any naive $G$-spectrum and define a genuine $G$-spectrum $i_*E$ by $(i_*E)_n= \widetilde{\TT}_G^n\smash E_n$ with bonding maps
$$
 \TT_G\smash i_*E_n\cong \widetilde{\TT}_G\smash T\smash i_*E_n\xrightarrow{\id\smash\sigma_n}\widetilde{\TT}_G^{n+1}\smash E_{n+1}.
$$
The resulting functor $i_*$ from naive to genuine $G$-spectra has a right adjoint $i^*$, which is defined by $(i^*E)_n=\underline{\Hom}_G(\widetilde{\TT}_G^n,E_n)$ with bonding maps
\begin{align*}
 &T\smash i^*E_n\to i^*E_{n+1}=\underline{\Hom}_G(\widetilde{\TT}_G^{n+1},E_{n+1}) \text{ adjoint to}\\
 &\widetilde{\TT}^{n+1}_G\smash T\smash i^*E_n\cong \TT_G\smash\widetilde{\TT}^{n}_G\smash  \underline{\Hom}_G(\widetilde{\TT}_G^n,E_n) \xrightarrow{ev}\TT_G\smash E_n\xrightarrow{\sigma_n} E_{n+1}.
\end{align*}

This way, we have defined a \emph{change of universe} adjunction
$$
i_*:\Nspec(\sPre_.(\Gsmk),T\smash-)\rightleftarrows\Nspec(\sPre_.(\Gsmk),\TT_G\smash-):i^*.
$$
The name is derived from an account to classical stable equivariant topology based on coordinate-free spectra, where spectra are indexed on a universe with a trivial $G$-action in the naive case and indexed on a universe of arbitrary representations in the genuine case.

\begin{lemma}\label{lemma:changeofuniverseisQuillen}
The change of universe adjunction $(i_*,i^*)$ is a Quillen adjunction with respect to the stable model structures.
\end{lemma}
\begin{proof}
The pair $(i_*,i^*)$ is a Quillen adjunction with respect to the levelwise model structures. Let $X$ be a stably fibrant genuine $G$-spectrum, in particular we have weak equivalences
$$
X_n\xrightarrow{\sim}\underline{\Hom}_G(\TT_G,X_{n+1})
$$
of $\AA^1$-locally fibrant simplicial presheaves for every $n$. The right Quillen functor $\underline{\Hom}_G(\widetilde{\TT}_G^n,-)$ preserves them and we compute
\begin{align*}
i^*X_n\cong\underline{\Hom}_G(\widetilde{\TT}_G^n,X_n)&\simeq\underline{\Hom}_G(\widetilde{\TT}_G^n,\underline{\Hom}_G(\TT_G,X_{n+1}))\\
&\cong \underline{\Hom}_G(T,\underline{\Hom}_G(\widetilde{\TT}^{n+1}_G,X_{n+1}))=(\Omega_Ti^*X)_n
\end{align*}
and note that $i^*X$ is a stably fibrant naive $G$-spectrum \cite[Lemma 2.7]{MSS}. Further, the adjunction $(i_*,i^*)$ is compatible with the simplicial enrichments and we combine this with the (SM7)-style characterization of stable equivalences \cite[Corollary 2.12]{MSS}: Let $W$ be a stably fibrant and levelwise-injective fibrant genuine $G$-spectrum and let $f:X\to Y$ be a trivial cofibration of naive $G$-spectra. The diagram
\begin{equation*}
\begin{xy}
\xymatrix{
\sset(i_*Y,W)\ar[r]^{i_*f^*}\ar[d]^\cong&\sset(i_*X,W)\ar[d]^\cong\\
\sset(Y,i^*W)\ar[r]^\sim_{f^*}&\sset(X,i^*W)
}
\end{xy}
\end{equation*}
commutes and therefore $i_*f$ is a stable equivalence (and a cofibration).
\end{proof}

The forgetful functor $(-)^e:\sPre(\Gsmk)\to\sPre(\smk)$ (the $e$-fixed points functor) also has a canonical prolongation
$$
(-)^e:\Nspec(\sPre_.(\Gsmk),T\smash-)\to\Nspec(\sPre_.(\smk),T\smash-)
$$
and for a (genuine) $G$-spectrum $E$, we call $E^e$ (resp.~$(i^*E)^e$) the underlying non-equivariant spectrum of $E$.

\begin{lemma}\label{lemma:naiveunit}
 Let $E$ be a naive $G$-spectrum. The unit morphism
 $$
 E\to i^*i_*E
 $$
 is a non-equivariant stable equivalence.
\end{lemma}
\begin{proof}
Let $X$ be a naive $G$-equivariant suspension spectrum. Consider the commutative diagram
\begin{equation}\label{eq:changeofuniverse}
\begin{xy}
\xymatrix{
X^e\ar[r]\ar[d]_\sim&i^*i_*X^e\ar[d]^\sim\\
R^\infty X^e\ar[r]&R^\infty i^*i_*X^e
}
\end{xy}
\end{equation}
of non-equivariant spectra. We compare domain and codomain of the lower horizontal morphism. The level $n$ in the domain is given by
\begin{align*}
 R^\infty X^e_n &= \colim_{j\geq 0}\underline{\Hom}(T^j,X^e_{j+n})\\
 &=\colim_{j\geq 0}\underline{\Hom}(T^j,T^j\smash X^e_{n})\\
 \intertext{while for the codomain we need a few transformations to compute}
 R^\infty i^*i_*X^e_n &= \colim_{j\geq 0}\underline{\Hom}(T^j,i^*i_*X^e_{j+n})\\
 &= \colim_{j\geq 0}\underline{\Hom}(T^j,\underline{\Hom}_G(\widetilde{\TT}_G^{j+n},\widetilde{\TT}_G^{j+n}\smash X_{j+n})^e)\\
 &= \colim_{j\geq 0}\underline{\Hom}_G(G_+\smash T^j\smash\widetilde{\TT}_G^{j+n},\widetilde{\TT}_G^{j+n}\smash X_{j+n})\\
\intertext{and replace $G_+\smash\widetilde{\TT}_G^{j+n}$ by the weakly equivalent $G_+\smash T^{(j+n)(|G|-1)}$. The equivariant weak equivalence is given by $G_+\smash Y^e\to G_+\smash Y, (g,x)\mapsto(g,g\cdot x)$ in $\sPre_.(\Gsmk)$. We continue}
 &\simeq \colim_{j\geq 0}\underline{\Hom}_G(G_+\smash T^{j+(j+n)(G-1)},\widetilde{\TT}_G^{j+n}\smash X_{j+n})\\
 &= \colim_{j\geq 0}\underline{\Hom}(T^{j+(j+n)(G-1)},(\widetilde{\TT}_G^{j+n}\smash X_{j+n})^e)\\
 &= \colim_{j\geq 0}\underline{\Hom}(T^{j+(j+n)(G-1)},T^{j+(j+n)(G-1)}\smash X_{n}^e).
\end{align*}
Thus, the (filtered and hence homotopy) colimit in the codomain is taken over a cofinal system for the colimit in the domain. Therefore, the lower horizontal morphism is a levelwise equivalence in diagram \eqref{eq:changeofuniverse}.

Now let $X$ be an arbitrary naive $G$-spectrum. $X$ is stably equivalent to the colimit
$$
\colim(\Sigma^\infty_TX_0\to\Sigma^\infty_TX_1[-1]\to\Sigma^\infty_TX_1[-1]\to\ldots)
$$
of shifted suspension spectra. By the same arguments as in \cite[Lemma 4.29]{MSS}, basically because stable weak equivalences are closed under filtered colimits \cite[Lemma 3.12]{MSS}, the conclusion follows from the first part of this proof.
\end{proof}

Not only the forgetful functor $(-)^e$ has a canonical prolongation, but also its space level adjoint functor $\ind=G_+\smash-$  prolongates canonically due to the twisting isomorphism $G_+\smash T\smash X\cong T\smash G_+\smash X$ to naive $G$-spectra.

\begin{lemma}\label{lemma:inducednaivequillen}
The adjunction
$$
\ind:\Nspec(\sPre_.(\smk),T\smash-)\rightleftarrows\Nspec(\sPre_.(\Gsmk),T\smash-):\res=(-)^e
$$
is a Quillen adjunction with respect to the stable model structures.
\end{lemma}
\begin{proof}
First, note that $(\ind,\res)$ is a Quillen adjunction for the levelwise model structures by Lemma \ref{lemma:HA1rightQuillen} and that $\res$ preserves levelwise equivalences. Since we have
\begin{align*}
 \res(R^\infty X)&= \res\left(\colim_{n\geq0}\underline{\Hom}_{G}(T^n,X_n)\right)\\
 &\cong \colim_{n\geq0}\res\left(\underline{\Hom}_{G}(T^n,X_n)\right)\\
 &\cong \colim_{n\geq0}\underline{\Hom}(T^n,\res(X_n))=R^\infty \res(X)\\
\end{align*}
it follows that $\res$ also preserves stable equivalences. Together with a characterization of stably fibrant objects \cite[Lemma 2.7\&2.8]{MSS} a similar computation reveals that $\res$ preserves stably fibrant objects. As the stable model structures are left Bousfield localizations of the levelwise ones, it is sufficient to show that $\ind$ maps trivial cofibrations to stable equivalences. So let $f:X\to Y$ be a trivial cofibration in $\Nspec(\sPre_.(\smk),T\smash-)$ and let $W$ be a stably fibrant and injective-levelwise fibrant object in $\Nspec(\sPre_.(\Gsmk),T\smash-)$. We make use of the simplicial structure and observe that the diagram
\begin{equation*}
\begin{xy}
\xymatrix{
\sset(\ind(Y),W)\ar[r]^{\ind(f)^*}\ar[d]^\cong&\sset(\ind(X),W)\ar[d]^\cong\\
\sset(Y,\res(W))\ar[r]^\sim_{f^*}&\sset(X,\res(W))
}
\end{xy}
\end{equation*}
commutes and that $\res(W)$ is still stably fibrant and 'injective'. Thus, $\ind(f)$ is a stable equivalence \cite[Corollary 2.12]{MSS}.
\end{proof}

\begin{lemma}\label{lemma:naivewhitehead}
Let $d:E\to F$ be a non-equivariant stable equivalence of naive $G$-spectra and let $X$ be stably equivalent to an induced naive $G$-spectrum. Then the map
$$
d_*:[X,E]\xrightarrow{\cong}[X,F] 
$$
is an isomorphism.
\end{lemma}
\begin{proof}
Due to naturality the diagram
\begin{equation*}
\begin{xy}
\xymatrix{
[X,E]\ar[d]^{d_*}\ar[r]^\cong&[\ind(D),E]\ar[d]^{d_*}\ar[r]^\cong&[D,E^e]\ar[d]^{d^e_*}\\
[X,F]\ar[r]^\cong&[\ind(D),F]\ar[r]^\cong&[D,F^e]\\
}
\end{xy}
\end{equation*}
commutes, where the maps decorated with '$\cong$' are isomorphisms by Lemma \ref{lemma:inducednaivequillen} and the assumption of a stable equivalence between $X$ and $\ind(D)$. Further, we assume the $d^e$ is a stable equivalence, hence $(d^e)_*$ and $d_*$ are isomorphisms.
\end{proof}

\begin{proposition}\label{prop:reducetonaive}
Let $X$ be stably equivalent to an induced naive $G$-spectrum and let $E$ be any naive $G$-spectrum. Then there is an isomorphism
$$
i_*:[X,E]\xrightarrow{\cong}[i_*X,i_*E]. 
$$
\end{proposition}
\begin{proof}
By Lemma \ref{lemma:naiveunit} and Lemma \ref{lemma:naivewhitehead} the morphism $i_*$ is a composition of isomorphisms
$$
i_* :[X,E]\xrightarrow{\eta_E}[X,i^*i_*E]\cong[i_*X,i_*E]. 
$$
\end{proof}

With the same arguments as for Lemma \ref{lemma:inducednaivequillen} all the other induction/restriction adjunctions
$$
\ind^H_G:\sPre_.(H\smk)\rightleftarrows\sPre(\Gsmk):\res^G_H
$$
prolongate to Quillen adjunctions between the respective naive equivariant categories as well. This is also true for the fixed-point functors and we record the following lemma for the study of fixed-point functors of genuine $G$-spectra in the next subsection.

\begin{lemma}\label{lemma:prolongfixedareQuillen}
For all $H\leq G$, the canonically prolongated adjunction
$$
(-)_H:\Nspec(\sPre_.(\smk),T\smash-)\rightleftarrows\Nspec(\sPre_.(\Gsmk),T\smash-):(-)^H
$$
is a Quillen adjunction with respect to the stable model structure on both sides.
\end{lemma}
\begin{proof}
 Again, note that $((-)_H,(-)^H)$ is a Quillen adjunction for the levelwise model structures. Let $f:X\to Y$ be a stable acyclic cofibration of non-equivariant spectra. We have to show that $f_H$ is a stable equivalence of naive $G$-spectra or equivalently that for all $n\in\NN$ and $K\leq G$ the morphism $R^\infty(f_H)_n^K$ is an $\AA^1$-local weak equivalence. Since we have
 $$
  \underline{\Hom}(T^i,X^K) \cong \Hom_{\sPre(G)}(T^i\smash G/K_+\smash \widetilde{(~)}_+,X)
  \cong \underline{\Hom}_G(T^i,X)^K 
 $$
we see that $(R^\infty(f_H)_n)^K\cong R^\infty((f_H)^K)_n$ holds and the statement follows from Lemma \ref{lemma:HA1rightQuillen}.
\end{proof}

\subsection{Characterization of Stable Weak Equivalences}

In this section we define two fixed point functors
\begin{align}
(-)^H &:\Nspec(\sPre_.(\Gsmk),\TT_G\smash-)\to\Nspec(\sPre_.(\smk),T\smash-)\\
\Phi^H&:\Nspec(\sPre_.(\Gsmk),\TT_G\smash-)\to\Nspec(\sPre_.(\smk),T\smash-)
\end{align}
from $G$-spectra to non-equivariant spectra for any subgroup $H\leq G$. The situation is pretty much the same as in classical stable equivariant homotopy theory, where the (Lewis-May) fixed point functor $(-)^H$ has the expected left adjoint, but is rather abstract and the \emph{geometric fixed point functor} $\Phi^H$ is the levelwise extension of the unstable fixed point functor.  We show that both families of fixed-point functors detect motivic equivariant stable weak equivalences. This means that we obtain two stable versions of Proposition \ref{prop:charofA1localequivariantweq}.

\subsubsection*{The Lewis-May fixed points}
For a non-equivariant $T$-spectrum $E$ we define the push forward $E\fixed$ to a genuine $G$-spectrum by the composition 
\begin{equation}\label{eq:factorfixed}
\begin{xy}
\xymatrix@C=-5mm{
\Nspec(\sPre_.(\smk),T\smash -)\ar[rr]^{(-)\fixed}\ar[dr]_{(-)\tr}&&\Nspec(\sPre_.(\Gsmk),\TT_G\smash-)\\
&\Nspec(\sPre_.(\Gsmk),T\smash -),\ar[ur]_{i_*}&
}
\end{xy}
\end{equation}
that is $X\fixed$ is the genuine $G$-equivariant spectrum defined by
$$
(X\fixed)_n = \widetilde{\TT}_G^n\smash (X_n)\tr
$$
where $\widetilde{\TT}_G$ is the representation sphere associated to the reduced regular representation and $(X_n)\tr$ is the image of $X_n$ under the left adjoint functor $(-)\tr$ from the adjunction
\begin{equation}\label{eq:trivialGactionKANextended}
 (-)\tr:\sPre_.(\smk)\rightleftarrows\sPre_.(\Gsmk):(-)^G
\end{equation}
of left Kan extensions, cf.~\eqref{eq:trivialGaction}.
The bonding maps of $X\fixed$ are defined by
\begin{equation*}
\begin{xy}
\xymatrix{
\TT_G\smash\widetilde{\TT}_G^{\smash n}\smash (X_n)\tr\ar@{-->}[r]\ar[d]_\cong^\tau&\widetilde{\TT}_G^{\smash n+1}\smash (X_{n+1})\tr\\
\widetilde{\TT}_G\smash\widetilde{\TT}_G^{\smash n}\smash T\smash (X_n)\tr\ar[ur]_{\id\smash\sigma_n}
}
\end{xy}
\end{equation*}

Since not only $((-)\tr,(-)^G)$, but by Lemma \ref{lemma:prolongfixedareQuillen} the whole family of fixed-point adjunctions canonically prolongates to Quillen adjunctions
$$
(-)_H:\Nspec(\sPre_.(\smk),T\smash-)\rightleftarrows \Nspec(\sPre_.(\Gsmk),T\smash-):(-)^H
$$
we may compose adjoints and make the following definition.

\begin{defi}
Let $X$ be a genuine $G$-equivariant spectrum. We define the (Lewis-May) $H$-fixed points of $X$ by 
$$
X^H := (i^*X)^H.
$$ 
\end{defi}

\begin{lemma}\label{lemma:lewismayfixedpoints}
 The adjunction
 $$
 (-)\fixed:\Nspec(\sPre_.(\smk),T\smash-)\rightleftarrows\Nspec(\sPre_.(\Gsmk),\TT_G\smash-):(-)^G
 $$
 as well as the other $H$-fixed point adjunctions are Quillen adjunctions with respect to the stable model structures.
\end{lemma}
\begin{proof}
 The Lewis-May fixed point adjunctions are compositions of Quillen adjunctions by Lemma \ref{lemma:changeofuniverseisQuillen} (change of universe) and Lemma \ref{lemma:prolongfixedareQuillen} (naive fixed points).
\end{proof}

\begin{proposition}\label{prop:charstableweqfixed}
 Let $f:X\to Y$ be a morphism in $\Nspec(\sPre(\Gsmk))$. Then the following are equivalent
\begin{enumerate}
 \item $f$ is a stable weak equivalence.
\item For all subgroups $H\leq G$, the morphism $f^H$ is a stable equivalence of non-equivariant spectra.
\end{enumerate}
\end{proposition}
\begin{proof}
 The morphism $f$ is a stable equivalence of $G$-spectra if and only if it induces isomorphisms on all weighted stable homotopy groups $\pi^H_{s,t}$. We compute
\begin{align*}
[G/H\smash S^{s+j}\smash (\AA[G]-0)^{t+j},X_j]^G&\cong [G/H\smash S^{s+j}\smash (\GG_m)^{t+j}\smash\widetilde{\TT}_G^{t+j},X_i]^G\\
\intertext{where we use Lemma \ref{lemma:TT_Gsplitting1} and the splitting $\TT_G=T\smash\widetilde{\TT}_G\simeq S^1\smash\GG_m\smash\widetilde{\TT}_G$, so that we can (cofinally) replace $\AA[G]-0$ by $\GG_m\smash\widetilde{\TT}_G$ and obtain}
&\cong [G/H\smash S^{s+j}\smash (\GG_m)^{t+j},\Omega_{\widetilde{\TT}_G}^{t+j}X_j]^G\\
&\cong [G/H\smash S^{s+j}\smash (\GG_m)^{t+j},i^*X[-t]_j]^G\\
&\cong [S^{s+j}\smash (\GG_m)^{t+j},i^*X[-t]_j]^H\\
&\cong [S^{s+j}\smash (\GG_m)^{t+j},i^*X[-t]_j^H]\\
\end{align*}
So that equivalently $f^H$ induces isomorphisms on non-equivariant weighted stable homotopy groups and hence is a stable equivalence for all $H\leq G$.
\end{proof}

\subsubsection*{The geometric fixed points}

We will need the following lemma to extend the adjunction of Corollary \ref{lemma:adjunction} from unstable to stable homotopy theories.

\begin{lemma}\label{lemma:fixedpointsofregrep}
The $G$-fixed points of the regular representation sphere are canonically isomorphic to the Tate object $T$, i.e.
$$
(\TT_G)^G\cong T
$$ 
\end{lemma}
\begin{proof}
The regular representation $\AA[G]$ decomposes into a sum $\oplus_{i=1}^nm_iU_i$ of inequivalent irreducible representations $U_i$. Let $U_1$ be the trivial representation, which splits off canonically due to the norm element $\Sigma_{g\in G} g$ in the finite group case. Then we have
$$
U_i^G\cong\begin{cases}
       \AA^1 & \text{ if } i=1\\
       0 & \text{ else,}\\
      \end{cases}
$$
because non-trivial fixed-points would give a $G$-invariant submodule and hence a $G$-invariant complement (by Maschke's Theorem in our case). 
\end{proof}
\begin{corollary}\label{cor:geomfixprolongate}
 There is a canonical natural isomorphism
$$
 (-\smash T)\circ (-)^G\rightarrow (-)^G\circ (-\smash\TT_G)
$$
of functors $\sPre_.(\Gsmk)\to\sPre_.(\smk)$ and hence a prolongation of the adjunction \eqref{eq:adjunction} to an adjunction
$$
\Phi^G:\Nspec(\sPre_.(\Gsmk),(-\smash\TT_G))\rightleftarrows\Nspec(\sPre_.(\smk),(-\smash T)).
$$
\end{corollary}
\begin{proof}
The left Kan extension $(-)^G$ from Corollary \ref{lemma:adjunction} preserves smash products since it is also right adjoint by Remark \ref{rem:fixedisright}.

Therefore, the isomorphism from the lemma above gives a natural isomorphism
$$
T\smash (-)^G \cong \TT_G^G\smash(-)^G\cong(\TT_G\smash - )^G.
$$

From this natural transformation $\tau:((-)^G\smash T)\xrightarrow{\cong} (-\smash\TT_G)^G$ one obtains a prolongation of $(-)^G$ by $(X_.)^G_n=(X_n)^G$ with bonding maps
\begin{equation*}   
\begin{xy}
\xymatrix{
T\smash X_n^G\ar[d]_{\tau_{X_n}}\ar@{-->}[r]&X_{n+1}^G\\
(\TT_G\smash X_n)^G\ar[ur]_{\sigma_n^G}&
}
\end{xy}
\end{equation*}
To prolongate the right adjoint $R^G$ of $(-)^G$ one needs a natural transformation 
$$
\TT_G\smash R^G(-)\to R_G(T\smash -),
$$
but using the adjunction and in particular the counit $\epsilon$ we obtain natural morphisms
$$
(\TT_G\smash R^G(-))^G\cong T\smash (R^G(-))^G\xrightarrow{\id\smash\epsilon}T\smash -.
$$
The prolongations are still adjoint.
\end{proof}

\begin{rem}
For a finite group $G$ the norm element $\Sigma_{g\in G} g\in \AA[G]$ gives a canonical splitting $\AA[G]\cong\AA^1\times\widetilde{\AA[G]}$ of the trivial part of the regular representation. Therefore, we have a canonical morphism from the Tate object $T$ with a trivial action to the regular representation sphere $\TT_G$ which factors for any $H\leq G$ as
\begin{equation*}   
\begin{xy}
\xymatrix{
T\ar[rr]\ar@{-->}[dr]_{c_H}&&\TT_G\\
&\TT_G^H\ar[ur].
}
\end{xy}
\end{equation*}
This canonical morphism $c_H$ gives a natural transformation
$$
T\smash (-)^H\xrightarrow{c_H}\TT_G^H\smash (-)^H\cong (\TT_G\smash -)^H
$$
which leads to a prolongation of the $H$-fixed points to a functor
\begin{equation}
\Phi^H:\Nspec(\sPre_.(\Gsmk),\TT_G\smash-)\to\Nspec(\sPre_.(\smk),T\smash-).
\end{equation}
\end{rem}

\begin{lemma}\label{lemma:geomfixedpointsnicecompatible}
 Let $X\in\psPre(\Gsmk)$ and let $Y$ be a genuine equivariant $G$-spectrum. For all subgroups $H\leq G$, we have
 $$
 \Phi^H(X\smash Y) = X^H\smash \Phi^H(Y).
 $$
 In particular, $\Phi^G$ is compatible with suspension spectra in the sense that 
 $$
 \Phi^G(\Sigma^\infty_{\TT_G}X) = \Sigma^\infty_TX^G.
 $$
\end{lemma}
\begin{proof}
The geometric fixed points functor $\Phi^H$ is a prolongation and smashing with a space is defined as a levelwise smash product, thus the first statement follows from the compatibility of the space level fixed point functors with smash products. For the second statement additionally use Lemma \ref{lemma:fixedpointsofregrep}.
\end{proof}

One adds a disjoint basepoint to the unique morphism $EG\to\ast$ and then takes the homotopy cofiber of the suspension spectra in $\Nspec(\Gsmk)$ to acquire the cofiber sequence
\begin{equation}\label{eq:fundamcofsequence}
EG_+\to S^0\xrightarrow{p} \widetilde{EG},
\end{equation}
which is of fundamental importance in equivariant homotopy theory.

\begin{lemma}\label{lemma:EGcontractunrecuded}
 The unreduced suspension $\widetilde{EG}$ defined by the cofiber sequence
 $$
 EG_+\to S^0\to \widetilde{EG}
 $$
 is non-equivariantly contractible.
\end{lemma}
\begin{proof}
 The space $EG$ is non-equivariantly contractible, hence the morphism of spectra $EG_+\to S^0$ is a stable weak equivalence of the underlying non-equivariant spectra. Applying \cite[Lemma 3.7]{MSS} twice to the long exact sequence of underlying $T$ spectra
 $$
 \ldots\to\pi_{t+1,s}(\widetilde{EG})\to\pi_{t,s}(EG_+)\xrightarrow{\cong}\pi_{t,s}(S^0)\to\pi_{t,s}(\widetilde{EG})\to\ldots
 $$
 we see that $\widetilde{EG}$ is contractible.
\end{proof}

\begin{lemma}\label{lemma:EGsmashnoneqeqiseq}
 Let $f:X\to Y$ be a non-equivariant stable equivalence of equivariant motivic spectra. Then 
 $$
 \id\smash f:EG_+\smash X\to EG_+\smash Y
 $$
 is an equivariant stable equivalence.
\end{lemma}
\begin{proof}
 We consider the cofiber sequence
$$
X\xrightarrow{f} Y \to hocofib(f)=:Z
$$
and assume that $Z$ is non-equivariantly contractible. Let $Z\to Z'$ be a stably fibrant replacement in $\Nspec(\psPre(\Gsmk),\TT_G\smash-)$. Then $Z'$ is levelwise non-equivariantly contractible and $EG_+\smash Z$ is stably equivalent to $EG_+\smash Z'$. But $EG_+\smash Z'$ is even equivariantly levelwise contractible and hence so is $EG_+\smash Z$.
\end{proof}

For a comparison of geometric and Lewis-May fixed points, we introduce the following generalization of $EG$. A \emph{family of subgroups} of $G$ is defined to be a set $\mathcal F$ of subgroups of $G$, such that $\mathcal F$ is closed under taking subgroups and conjugation. Given such a family $\mathcal F$, there might exist a $G$-representation $V=V_{\mathcal F}$ with the property that
\begin{equation}\label{eq:famofsubgroups}
 V^H\text{ is }\begin{cases}
         >0 & \text{ if } H\in\mathcal F,\\
         0 &\text{ if } H \not\in \mathcal F.
        \end{cases}
\end{equation}
On the other hand, given a $G$-representation $V$, the set of subgroups with defining property \eqref{eq:famofsubgroups} is a family of subgroups. We consider the cofiber sequence
$$
(V-0)_+\to S^0\to S^V
$$
and observe that the fixed points $(S^V)^H$ are computed by the diagram
\begin{equation*}
\begin{xy}
\xymatrix{
(V-0)^H\ar[r]\ar[d]&V^H\ar[d]\\
\ast\ar[r]&(S^V)^H.
}
\end{xy}
\end{equation*}
Thus, $(S^V)^H$ is $S^0$ for subgroups $H$ which are not in $\mathcal F$ and otherwise $(S^V)^H$ is equal to $S^{2r,r}$, for some $r>0$. Denote by $E\mathcal F$ the infinite smash product $\colim_{j\geq 0} (V-0)^{\smash j}$ and by $\widetilde{E\mathcal F}$ the infinite smash product $\colim_{j\geq 0} (S^V)^{\smash j}$. It follows that $\widetilde{E\mathcal F}^H$ is $S^0$ if $H$ is not in $\mathcal F$. For a subgroup $H\in\mathcal F$ the $H$-fixed points are an infinite smash of positive dimensional spheres and therefore contractible. In particular, we note that for the family $\mathcal P$ of all proper subgroups of $G$, the reduced regular representation gives an adequate representation and the fixed points of the unreduced suspension $\widetilde{E\mathcal P}$ are given by
$$
\widetilde{E\mathcal P}^H\simeq\begin{cases}
         \ast & \text{ if } H<G\\
         S^0 &\text{ if } H =G.
        \end{cases}
$$

\begin{lemma}\label{lemma:fixedvsgeometric}
 The evaluation morphism
 $$
 (\widetilde{E\mathcal P}\smash X)^G\to \Phi^G(X)
 $$
 is a levelwise equivalence of non-equivariant spectra.
\end{lemma}
\begin{proof}
We compute that
\begin{align*}
  (\widetilde{E\mathcal P}\smash X)^G_n &= \underline{\Hom}_G(\widetilde{\TT}_G^n,\widetilde{E\mathcal P}\smash X_n)^G\\
  &\cong \sset_G(\widetilde{\TT}_G^n\smash (-)\tr,\widetilde{E\mathcal P}\smash X_n)\\
  \intertext{where $\widetilde{\TT}_G^n$ is a homotopy colimit of equivariant cells $G/H_+\smash S^{p_H,q_H}$ and therefore}
  &\cong \holim_{H\leq G}\sset_G(G/H_+\smash S^{p_H,q_H}\smash (-)\tr,\widetilde{E\mathcal P}\smash X_n)\\
  &\cong \holim_{H\leq G}\sset(S^{p_H,q_H}\smash (-)\tr,\widetilde{E\mathcal P}^H\smash X_n^H)\\
  \intertext{All the non-initial holim-factors corresponding to proper subgroups are contractible and since $\widetilde{\AA[G]}$ has no trivial subrepresentation we have $(p_G,q_G)=(0,0)$, so that}
  &\simeq \sset(S^{0}\smash (-)\tr,\widetilde{E\mathcal P}^G\smash X_n^G)\cong X_n^G.
\end{align*}
\end{proof}

We are now ready for a characterization of equivariant stable equivalences by their geometric fixed points. 

\begin{proposition}\label{prop:charstableweqgeometric}
 Let $f:X\to Y$ be a morphism in $\Nspec(\sPre_.(\Gsmk))$. Then the following are equivalent
\begin{enumerate}
\item $f$ is a stable weak equivalence.
\item For all subgroups $H\leq G$, the morphism $\Phi^H(f)$ is a stable equivalence of non-equivariant spectra.
\end{enumerate}
\end{proposition}
\begin{proof}
Assume that $f$ is a stable equivalence. Let $\mathcal P_H$ be the family of all proper subgroups of $H$. When applying the left Quillen functor $\widetilde{E\mathcal P_H}\smash -$ we still have a stable equivalence and by Proposition \ref{prop:charstableweqfixed} for all subgroups $H$ of $G$ we thus have a non-equivariant stable equivalence
$$
(\widetilde{E\mathcal P_H}\smash f)^H:(\widetilde{E\mathcal P_H}\smash X)^H\to(\widetilde{E\mathcal P_H}\smash Y)^H
$$
which implies by Lemma \ref{lemma:fixedvsgeometric} that $\Phi^H(f)$ is a stable equivalence.

Conversely, assume that for all subgroups $H$ of $G$ the map $\Phi^H(f)$ on geometric fixed points is a stable equivalence. We proceed by induction on the order of $G$. For $|G|=1$ there is nothing to show, since $\Phi^G$ is basically the identity then. So let $G$ be non-trivial and assume the claim to be true for all proper subgroups of $G$. So $\res^G_Hf$ is an equivariant stable equivalence for all proper subgroups $H$ of $G$ and by Proposition \ref{prop:charstableweqfixed} this implies that for these subgroups also $f^H$ is a non-equivariant stable equivalence. We are going to show that $f^G$ is a stable equivalence as well. Smashing $f$ with the norm sequence \eqref{eq:fundamcofsequence} for $E\mathcal P$ we obtain a diagram
\begin{equation*}
\begin{xy}
\xymatrix{
E\mathcal P\smash X\ar[d]\ar[r] & X\ar[r]\ar[d]^f & \widetilde{E\mathcal P}\smash X\ar[d]\\
E\mathcal P\smash Y\ar[r]& Y\ar[r] & \widetilde{E\mathcal P}\smash Y
}
\end{xy}
\end{equation*}
where $E\mathcal P_+\smash f$ is a stable equivalence by an argument completely analogous to the proof of Lemma \ref{lemma:EGsmashnoneqeqiseq}. We may apply $(-)^G$ to the whole diagram above and using Lemma \ref{lemma:fixedvsgeometric} we find that $f^G$ is surrounded by stable equivalences in the diagram
\begin{equation*}
\begin{xy}
\xymatrix{
(E\mathcal P\smash X)^G\ar[d]^\sim\ar[r] & X^G\ar[r]\ar[d]^{f^G} & (\widetilde{E\mathcal P}\smash X)^G\ar[d]^\sim\\
(E\mathcal P\smash Y)^G\ar[r]& Y^G\ar[r] & (\widetilde{E\mathcal P}\smash Y)^G
}
\end{xy}
\end{equation*}
with rows cofiber sequences. 
Therefore, $f^G$ is a stable equivalence and we conclude again by using Proposition \ref{prop:charstableweqfixed} that $f$ itself is a stable equivalence.
\end{proof}

\section{Representability of Equivariant Algebraic K-Theory}

This subsection starts with a recollection of equivariant algebraic K-theory following Thomason \cite{thomason1987algebraic}. The main result of this subsection shows that equivariant algebraic K-theory does not satisfy descent with respect to topologies that, like the $H$-Nisnevich topology, contain certain morphisms as coverings.
We also recall a result of Krishna and Østvær \cite{krishna2012nisnevich} that the equivariant Nisnevich topology of Definition \ref{def:equiNistop} allows K-theory to satisfy descent.
Finally, we discuss the effect of our non-descent result on the K-theory descent property of the isovariant Nisnevich topology as it is investigated in \cite{serpe2010descent}.

\begin{defi} Let $X$ be in $\Gsmk$. A quasi-coherent $G$-module $(F,\varphi)$ on $X$ is given by a quasi-coherent $\mathcal O_X$-module $F$ and an isomorphism
$$
\varphi:\alpha_X^*F\xrightarrow{\cong}pr_2^*F
$$
of $\mathcal O_{G\times X}$-modules, such that the cocycle condition
$$
(pr_{23}^*\varphi)\circ((\id\times\alpha_X)^*\varphi) = (m\times\id)^*\varphi
$$
is satisfied. $F$ is called coherent (resp.~locally free) if it is coherent (resp.~locally free) as an $\mathcal O_X$-module.
\end{defi}

Coherent $G$-modules on some $X$ in $\Gsmk$ form an abelian category $M(G,X)$ and locally free coherent $G$-modules ($G$-equivariant vector bundles) form an exact subcategory $P(G,X)$. To these exact categories we associate the simplicial nerve $BQM(G,X)$ (resp.~$BQP(G,X)$) of Quillen's $Q$-construction. Finally, denote by $\mathcal G(G,X)=\Omega BQM(G,X)$ and $K(G,X)=\Omega BQP(G,X)$ the K-theory spectra (or infinite loop spaces) associated to the exact categories of coherent $G$-modules on $X$ and to those that are locally free. In his fundamental work Thomason already shows that for a separated noetherian regular $G$-scheme $X$ the inclusion of categories induces an equivalence $K(G,X)\xrightarrow{\sim}\mathcal G(G,X)$ \cite[Theorem 5.7]{thomason1987algebraic} and that hence for such an $X$ the equivariant K-theory satisfies homotopy invariance in the sense that the projection induces an equivalence
$$
K(G,X)\to K(G,X\times \AA^n)
$$
even with respect to any linear $G$-action on $\AA^n$ \cite[Corollary 4.2]{thomason1987algebraic}.

By the origin of the use of the word \emph{motivic} in this area of mathematics, or in other words by Grothendieck's idea of what it should mean to associate a motive to a scheme, it should be considered a fundamental test for any candidate of a motivic homotopy category, whether it allows representability for a sufficient amount of cohomological theories or not. One obstacle for a theory $F$ to be representable in $\mathcal{H}(k,G)$ is that it has to satisfy (hypercover) descent with respect to the topology used to define the local model structure. This is a kind of homotopical sheaf condition which implies the compatibility of the theory $F$ with local weak equivalences. For the following we may restrict our attention to the weaker notion of \v Cech descent.

\begin{defi} An objectwise fibrant simplicial presheaf $F$ on a site $\C$ \emph{satisfies \v Cech descent} with respect to the topology on $\C$ if for any covering family $\{U_i\to X\}_i$ in $\C$ the morphism
\begin{equation}\label{eq:cech}
\begin{xy}
\xymatrix{
F(X)\ar[r]&\holim ( \prod_iF(U_i)\ar@<0.5ex>[r]\ar@<-0.5ex>[r]&\prod_{i,j}F(U_i\times_XU_j)\ar@<0.7ex>[r]\ar@<-0.7ex>[r]\ar[r]&\ldots)
}
\end{xy}
\end{equation}
is a weak equivalence of simplicial sets. An arbitrary simplicial presheaf is said to satisfy \v Cech descent if an objectwise fibrant replacement of it does.
\end{defi}

It is a straight reformulation of this definition that a simplicial presheaf $F$ satisfies \v Cech descent if and only if for any covering $\mathcal U=\{U_i\to X\}_i$ and an injective fibrant replacement $F'$ of $F$ the induced map
$$
\sset(X,F')\to \sset(\check C(\mathcal U),F')
$$
is a weak equivalence of simplicial sets.

In \cite[Theorem 5.4]{krishna2012nisnevich} Krishna and Østvær show that the presheaf of K-theory of perfect complexes on Deligne-Mumford stacks satisfies descent with respect to a version of the Nisnevich topology. Restricting the results from Deligne-Mumford stacks to the subcategory of $G$-schemes, the topology restricts to the equivariant Nisnevich topology and their results imply descent of equivariant K-theory for the equivariant Nisnevich topology (cf.~\cite[Remark 7.10]{krishna2012nisnevich}).

However, the rest of this section is devoted to showing that equivariant K-theory does not satisfy descent with respect to certain topologies, including the $H$-Nisnevich topology.

\begin{proposition}\label{prop:eqkthydescent}
Equivariant algebraic K-theory does not satisfy descent with respect to the $H$-Nisnevich topology.
\end{proposition}
\begin{proof}
Suppose that $K(G,-)$ satisfies descent for the $H$-Nisnevich topology on $\ZZ/2-\sm{\RR}$, then in particular $K(\ZZ/2,-)$ satisfies Cech descent for $\spec(\CC)_{\mathrm{gal}}$ and the $H$-Nisnevich cover
 $$
 \ZZ/2\times\spec(\CC)\tr\to\spec(\CC)_{\mathrm{gal}}
 $$
 induces a weak equivalence
$$
K(\ZZ/2,\spec(\CC)_{\mathrm{gal}})\xrightarrow{\sim} \holim \left( K(G,\ZZ/2\times \CC) \rightrightarrows K(\ZZ/2,(\ZZ/2\times\CC)^{\times 2})\ldots\right )
$$
as in \eqref{eq:cech}. We compute the equivariant K-theory of $G$-torsors using \cite[Proposition 3]{merkurjev2equivariant} as $K(\ZZ/2,\spec(\CC)_{\mathrm{gal}})\simeq K(\spec(\RR))$, $K(\ZZ/2,\ZZ/2\times\spec(\CC))\simeq K(\spec(\CC))$, and so on, which implies an equivalence
$$
K(\spec(\RR))\to \holim \left( K(\spec(\CC)) \rightrightarrows K(\spec(\CC)\times\spec(\CC)))\ldots\right ).
$$
Thus, the homotopy limit on the right hand side computes to
\begin{align*}
&\holim \left( K(\spec(\CC)) \rightrightarrows K(\spec(\CC)\times\spec(\CC)))\ldots\right )\\
&\simeq Map(\hocolim_n\check C(\ZZ/2\to\ast)_n ,K(\spec(\CC)))\\
&\simeq Map(EG,K(\spec(\CC))) \\
&=  K(\spec(\CC))^{hG}\simeq K^{et}(\spec(\RR)),
\end{align*}
and so we finally obtain an equivalence $K(\spec(\RR))\to K^{et}(\spec(\RR))$ which gives a contradiction, since $K^{et}(\spec(\RR))$ contains a non-zero additional information coming from the Brauer group.
\end{proof}

\begin{rem}~
 \begin{enumerate}
  \item The proof above can easily be generalized to more general field extensions. One needs to assure that there is some non-zero $l$-torsion in the Brauer group of the base field and that the $l$-completed descent spectral sequence (cf.~\cite[Corollary 1.5]{mitchell1997hypercohomology}) converges and hence allows to detect this additional $l$-torsion elements.
  \item The same proof also provides a counterexample to the main theorem of \cite{serpe2010descent} that equivariant K-theory satisfies 'isovariant' descent. In loc.~cit.~a parametrized version of scheme-theoretic isotropy is introduced as $G_X$, where $X$ is a $G$-scheme, and defined as the pullback
\begin{equation*}
\begin{xy}
\xymatrix{
G_X\ar[r]\ar[d]&G\times X\ar[d]^{\alpha\times\mathrm{pr}_X}\\
X\ar[r]^\Delta&X\times X.
}
\end{xy}
\end{equation*}
Now Serpé calls a family $\{U_i\to X \}_i$ in $\Gsmk$ an isovariant Nisnevich cover if the underlying family of schemes is a Nisnevich cover and for all $U_i\to X$ the induced morphism $G_{U_i}\to G_X$ furnishes a pullback square 
\begin{equation}\label{eq:isovariant}
\begin{xy}
\xymatrix{
G_{U_i}\ar[r]\ar[d]&G_X\ar[d]\\
U_i\ar[r]&X.
}
\end{xy}
\end{equation}
The singleton $\{f:\ZZ/2\times\spec(\CC)\tr\to\spec(\CC)_{\mathrm{gal}} \}$ defines an isovariant Nisnevich cover. This is because firstly the $G$-actions on domain and codomain are free. Therefore, the corresponding commutative square of type \eqref{eq:isovariant} is a pullback square. Secondly, $f$ is a non-equivariant Nisnevich covering, since the components of $G\times \spec(L)$ map to $\spec(L)$ along the elements of the Galois group.

Eventually, $\{f\}$ is also a counterexample to the proof of \cite[Proposition 2.7]{serpe2010descent}, since $f/G$ is the canonical map $\spec(L)\to\spec(k)$ which is not a Nisnevich cover.
\end{enumerate}
\end{rem}

\bibliographystyle{amsalpha} 
\bibliography{../ref}
\end{document}